%% file: fillings-of-virtually-overtwisted-torus-bundles.tex
\pdfoutput=1
\documentclass{amsart}
\usepackage{research}

\usetikzlibrary{decorations.markings}
\graphicspath{{./graphics/}}
\makeatletter
\def\input@path{{./graphics/}}
\makeatother
\usepackage{multirow}

\newcommand{\slope}{s}
\DeclareMathOperator{\Int}{Int}
\DeclareMathOperator{\tr}{tr}

\begin{document}
\title{On symplectic fillings of virtually overtwisted torus bundles}
\author{Austin Christian}
\begin{abstract}
We use Menke's JSJ-type decomposition theorem for symplectic fillings to reduce the classification of strong and exact symplectic fillings of virtually overtwisted contact structures on torus bundles to the same problem for tight lens spaces.  For virtually overtwisted structures on elliptic or parabolic torus bundles, this gives a complete classification.  For virtually overtwisted structures on hyperbolic torus bundles, we show that every strong or exact filling arises from a filling of a tight lens space via round symplectic 1-handle attachment, and we give a condition under which distinct tight lens space fillings yield the same torus bundle filling.
\end{abstract}
\maketitle

\section{Introduction}\label{sec:intro}
A rich source of contact manifolds is provided by symplectic manifolds-with-boundary.  When a symplectic manifold-with-boundary $(W,\omega)$ admits a vector field $Z$ on some neighborhood of $\partial W\subset W$ satisfying $Z\pitchfork\partial W$ and $\mathcal{L}_Z\omega=\omega$, there is naturally a contact structure on $M:=\partial W$, given by $\xi:=\ker\iota_Z\omega$.  In this case, we call $(W,\omega)$ a \emph{strong symplectic filling} of $(M,\xi)$.  This leads to questions of geography and botany.  Namely, with a contact manifold $(M,\xi)$ in hand, we wonder whether a symplectic filling $(W,\omega)$ of $(M,\xi)$ exists and, if so, whether the collection of such fillings can be enumerated.\\

It has long been known that not all contact manifolds admit symplectic fillings.  For instance, in dimension three a necessary condition for a contact manifold to be fillable is that its structure be tight (\cite{eliashberg1990filling},\cite{gromov1985pseudo}), and in fact this condition is not sufficient (\cite{etnyre2002tight}).  Among those contact manifolds which are fillable, some have been found to have unique fillings up to natural notions of equivalence --- e.g., $(S^3,\xi_{std})$ in \cite{eliashberg1990filling},\cite{gromov1985pseudo} --- while others admit infinitely many non-diffeomorphic fillings (c.f. \cite{ozbagci2004contact}).\\

One family of contact 3-manifolds for which the geography and botany questions have seen partial or complete answers is tight lens spaces.  A corollary of the classification of tight contact structures on lens spaces (\cite{giroux2000structures}, \cite{honda2000classification}) is that all of these tight structures admit exact fillings.  Work of McDuff \cite{mcduff1991symplectic} and Plamenevskaya-Van Horn-Morris \cite{plamenevskaya2010planar} has shown that this filling is unique when the lens space has diffeomorphism type $L(p,1)$ --- with the exception of the standard tight structure on $L(4,1)$, which admits two exact fillings.  The corresponding classification for strong fillings holds up to symplectic deformation equivalence and blowup.  In \cite{lisca2008symplectic}, Lisca catalogs the symplectic fillings of the standard contact structure on all 3-dimensional lens spaces.  Fossati has recently given topological constraints on the fillings of more general structures (\cite{fossati2019topological}), as well as a classification up to homeomorphism for fillings of lens spaces which arise from Legendrian surgery on Legendrian representatives of the Hopf link $H\subseteq S^3$ (\cite{fossati2019contact}).  While a complete classification of fillings for tight lens spaces remains open, much progress has been made.

\subsection{Torus bundles}\label{sec:intro-torus-bundles}
In this paper we reduce the classification problem for virtually overtwisted contact structures on torus bundles to the same problem for tight lens spaces.  In \cite{honda2000classification2}, Honda classified the tight contact structures on all torus bundles over $S^1$, and there is a dichotomy between those which are \emph{universally tight} --- meaning that the contact structure remains tight when lifted to the universal cover of the torus bundle --- and those which are \emph{virtually overtwisted} --- meaning that the torus bundle admits a finite-sheeted covering space for which the lifted contact structure is overtwisted.  Our methods apply exclusively to virtually overtwisted contact structures.\\

Given an element $A\in SL(\mathbb{Z},2)$, we define the mapping torus
\[
M_A := (\mathbb{R}^2/\mathbb{Z}^2\times I)/\sim,
\]
where $(\mathbf{x},1)\sim(A\mathbf{x},0)$.  We call $M_A$ the \emph{torus bundle with monodromy $A$}, and notice that the diffeomorphism type of $M_A$ depends only on the conjugacy class of $A$ in $SL(2,\mathbb{Z})$.  We will say that $M_A$ is
\begin{itemize}
	\item \emph{elliptic} if $|\tr~A|<2$;
	\item \emph{parabolic} if $|\tr~A|=2$;
	\item \emph{hyperbolic} if $|\tr~A|>2$.
\end{itemize}
This trichotomy will be used in the statement of our classification result, as will generators
\[
S = \left(\begin{matrix} 0 & 1\\ -1 & 0\end{matrix}\right)
\quad\text{and}\quad
T = \left(\begin{matrix} 1 & 1\\ 0 & 1 \end{matrix}\right)
\]
of $SL(2,\mathbb{Z})$.  It is known (c.f. \cite[Lemma 2.1]{honda2000classification2}) that every conjugacy class of $SL(2,\mathbb{Z})$ can be represented by one of
\begin{itemize}
	\item $A=\pm S$;
	\item $A=\pm T^{-1}S,\pm(T^{-1}S)^2$;
	\item $A=\pm T^n$, $n\in\mathbb{Z}$;
	\item $A=\pm T^{r_0}ST^{r_1}S\cdots T^{r_k}S$, $r_i\leq -2$, $r_0<-2$.
\end{itemize}
In the last case, the choice of conjugacy class representative is not unique, but in all cases the $\pm$ sign is unique.  We will refer to the monodromy $A$ as \emph{positive} or \emph{negative} depending on sign of its conjugacy class representative(s).\\

Ours is not the first study of symplectic fillings for contact torus bundles.  In \cite{ding2001symplectic}, Ding-Geiges showed that every torus bundle admits an infinite family of weakly-but-not-strongly symplectically fillable contact structures, all of which are universally tight.  In \cite{bhupal2014canonical}, Bhupal-Ozbagci show that for certain parabolic and hyperbolic torus bundles, there are precisely two isotopy classes of Stein fillable contact structures, and along the way they construct a Stein filling for every tight contact structure on a torus bundle with positive hyperbolic monodromy.  Golla-Lisca, in \cite{golla2015stein}, construct Stein fillable structures on a large family of torus bundles, with many of these structures being universally tight.  One of these constructions provides a Stein filling for each virtually overtwisted tight contact structure on a torus bundle with negative hyperbolic monodromy.  In \cite{ding2018strong}, Ding-Li consider the question of strong symplectic fillability for some contact structures on negative parabolic and negative hyperbolic torus bundles, and construct a Stein filling for the negative parabolic torus bundle with monodromy $-T^n$, $n\leq -1$.\\

In the present article, we use Menke's JSJ decomposition (\cite{menke2018jsj}) for symplectic fillings to obtain a classification result for fillings of virtually overtwisted contact structures on torus bundles.  Combined with (non-)existence results previously established in \cite{bhupal2014canonical}, \cite{golla2015stein}, and \cite{ding2018strong}, our main result is the following.

\begin{theorem}\label{main-theorem}
Let $M$ be a 3-dimensional torus bundle, $\xi$ a virtually overtwisted tight contact structure on $M$.
\begin{enumerate}[label=(\Alph*)]
	\item If $M$ is an elliptic torus bundle, then $(M,\xi)$ is not strongly symplectically fillable.\label{part:elliptic}
	\item If $M$ is the positive parabolic torus bundle with monodromy $T^n$ ($n\geq 2$), then $(M,\xi)$ is not strongly symplectically fillable.  If $M$ is any other parabolic torus bundle, then $(M,\xi)$ admits a unique strong filling up to symplectic deformation equivalence and blowup.
	\item If $M$ is a hyperbolic torus bundle, then there is a nonempty, finite list $(L(p_1,q_1),\xi_1),\ldots,(L(p_m,q_m),\xi_m)$ of tight lens spaces such that every strong (exact) symplectic filling of $(M,\xi)$ can be obtained from a strong (exact) symplectic filling of $(L(p_i,q_i),\xi_i)$, for some $1\leq i\leq m$, by attaching a round symplectic 1-handle.  In particular, $(M,\xi)$ is fillable.\label{part:hyperbolic}
\end{enumerate}
\end{theorem}

\begin{remark}$ $
\begin{enumerate}
	\item As noted above, some of the existence statements in Theorem \ref{main-theorem} are not new.  Indeed, \ref{part:elliptic} is weaker than \cite[Theorem 1.1]{etnyre2002tight}, which says that virtually overtwisted structures on elliptic torus bundles are not even \emph{weakly} symplectically fillable.  For the parabolic torus bundles, Ding-Li established the existence of a symplectic filling for the unique virtually overtwisted torus bundle with monodromy $-T^n$, $n\leq -1$, in \cite{ding2018strong}.  The existence statement for hyperbolic torus bundles was established in \cite{bhupal2014canonical} and \cite{golla2015stein}.
	
	\item The proof of Theorem \ref{main-theorem} produces explicit attaching data for the round symplectic 1-handles that appear in \ref{part:hyperbolic}.  Namely, we identify Legendrian knots $L^i_-,L^i_+$ in $(L(p_i,q_i),\xi_i)$ so that attaching a round symplectic 1-handle to a filling of $(L(p_i,q_i),\xi_i)$ along standard neighborhoods of $L^i_-$ and $L^i_+$ yields a filling of the torus bundle $(M,\xi)$, for each $1\leq i\leq m$.  Moreover, for $1\leq i\leq m-1$, we obtain $(L(p_{i+1},q_{i+1}),\xi_{i+1})$ from $(L(p_i,q_i),\xi_i)$ by performing $(+1)$-surgery along $L^i_+$ and $(-1)$-surgery along $L^i_-$.
	
	\item While the tight contact structure $\xi$ is virtually overtwisted, the structures $\xi_1,\ldots,\xi_m$ could be either universally tight or virtually overtwisted.
\end{enumerate}
\end{remark}

In Section \ref{sec:hyperbolic} we will give an explicit construction of the list of lens spaces specified by \ref{part:hyperbolic} of Theorem \ref{main-theorem}.  We will see that if $M$ has monodromy
\begin{equation}\label{eq:monodromy}
A = \pm T^{r_0}ST^{r_1}S\cdots T^{r_k}S,
\end{equation}
with $r_0\leq -3$ and $r_i\leq -2$ for $1\leq i\leq k$, then each lens space in the list has the diffeomorphism type of $L(p,q)$, where
\[
-\frac{p}{q} = [r_k,\ldots,r_1]
\quad\text{or}\quad
-\frac{p}{q} = [r_{j-1},r_{j-2},\ldots,r_0,r_k,r_{k-1},\ldots,r_{j+1}]
\]
for some $1\leq j\leq k$.  By $[a_0,\ldots,a_k]$ we mean
\[
[a_0,\ldots,a_k] = a_0-\cfrac{1}{a_1-\cfrac{1}{a_2-\cdots\cfrac{1}{a_k}}}.
\]

This list of lens spaces (along with their contact structures) can be encoded diagramatically as follows.  Per \cite{bhupal2014canonical} and \cite{golla2015stein}, the standard filling of a tight torus bundle with monodromy given by (\ref{eq:monodromy}) is obtained from that of $(S^1\times S^2,\xi_{std})$ by attaching Weinstein 2-handles along the link $\Lambda\subset S^1\times S^2$ depicted in Figure \ref{fig:filling}.  (The knots in this link are stabilized according to the contact structure $\xi$.)  To each $1\leq i\leq k$ we associate a tight lens space $(L_i,\xi_i)$, with a filling obtained from Figure \ref{fig:filling} by erasing the 1-handle and the unknot $K_i$.\\

\begin{figure}
\centering
\def\svgwidth{0.9\textwidth}
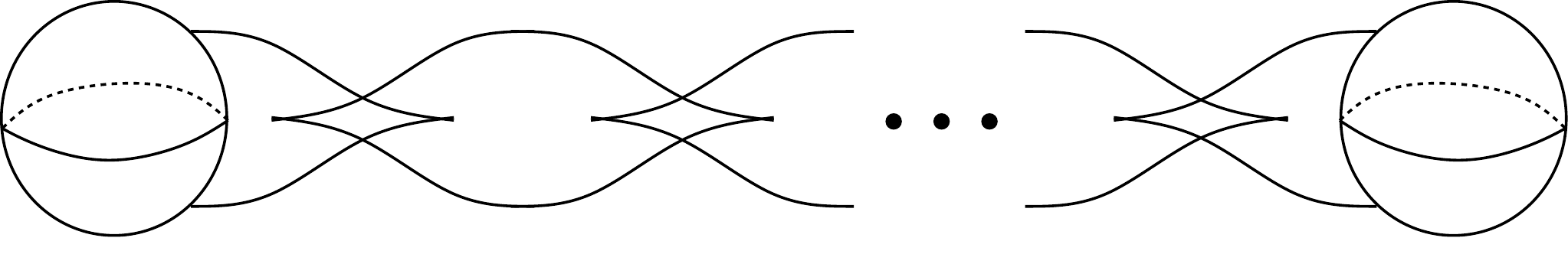
\caption{The natural filling of $(M,\xi)$, which has positive monodromy with coefficients $(r_0,r_1,\ldots,r_k)$, is obtained by attaching Weinstein 2-handles to the unique filling of $(S^1\times S^2,\xi_{std})$ along the Legendrian knots $K_0,\ldots,K_k\subset(S^1\times S^2,\xi_{std})$.  For each $0\leq i\leq k$, the knot $K_i$ is stabilized to ensure that $tb(K_i)=r_i+1$, and that $r(K_i)$ is determined by the contact structure $\xi$.  If the monodromy of $M$ is negative, we modify $\Lambda$ by requiring $lk(K_0,K_k)=-1$.}
\label{fig:filling}
\end{figure}

Now let
\[
0 = i_1 < i_2 < \cdots < i_m \leq k
\]
be precisely those indices for which $r_{i_j}\leq -3$.  Then $(M,\xi)$ decomposes into $m$ \emph{continued fraction blocks} --- a notion we will recall in Section \ref{sec:background} --- with a continued fraction block associated to each $i_j$.  Because $\xi$ is virtually overtwisted, $(M,\xi)$ admits a \emph{mixed torus} $T\subset (M,\xi)$, and we use $T$ to build a list $\mathcal{L}$ of tight lens spaces from whose fillings the fillings of $(M,\xi)$ will be obtained.  If $T$ is interior to the continued fraction block associated to $i_j$, for some $1\leq j\leq m$, then $\mathcal{L}$ has just one element --- the lens space $(L_{i_j},\xi_{i_j})$.  Otherwise, $T$ sits at the intersection of two continued fraction blocks, say, associated to $i_j$ and $i_{j+1}$ for some $1\leq j\leq m$.  (Here $i_{m+1}=i_1$.)  In this case
\[
\mathcal{L} = ((L_{i_j},\xi_{i_j}),(L_{i_j+1},\xi_{i_j+1}),\ldots,(L_{i_{j+1}},\xi_{i_{j+1}})).
\]
In words, each lens space in $\mathcal{L}$ is obtained by deleting the 1-handle and an unknot from Figure \ref{fig:filling}.  For each continued fraction block that $T$ meets, $\mathcal{L}$ contains the corresponding lens space, as well as the intermediate lens spaces obtained by deleting $(-2)$-framed unknots.  Our result says that every filling of $(M,\xi)$ is obtained from a filling of some $(L_i,\xi_i)\in\mathcal{L}$ by attaching a round 1-handle.  For the standard filling of $(L_i,\xi_i)$, obtained from Figure \ref{fig:filling} as described, this round 1-handle is attached along knots $K'_i,K''_i\subset(L_i,\xi_i)$ obtained from Figure \ref{fig:filling} by replacing each of the 3-balls with a cusp.  See Figure \ref{fig:lens-space-filling}.\\

If $tb(K_i)\neq -1$, there remains the ambiguity of how we distribute the stabilizations of $K_i$ among $K'_i$ and $K''_i$, but this is resolved by the mixed torus we are using.  Namely, the continued fraction block associated to $K_i$ decomposes into basic slices, each of which we think of as either positive or negative, and the stabilizations of $K_i$ correspond to these basic slices.  (We will recall the notion of a basic slice in Section \ref{sec:background}.)  The stabilizations of $K_i$ are then distributed among $K'_i$ and $K''_i$ according to the number of basic slices on either side of our mixed torus.  In particular, if the mixed torus lies at the boundary of two continued fraction blocks, then  all of the stabilizations of $K_i$ will lie on one of $K'_i$ or $K''_i$, with the other being a Legendrian unknot with Thurston-Bennequin number $-1$.  In all cases, attaching a round symplectic 1-handle along $K'_i$ and $K''_i$ corresponds to attaching a Weinstein 1-handle along $p'_i\in K'_i$ and $p''_i\in K''_i$, followed by a Weinstein 2-handle attachment along $K_i$.  The resulting filling of $(M,\xi)$ is the standard filling depicted in Figure \ref{fig:filling}.

\begin{figure}
\centering
\def\svgwidth{0.9\textwidth}
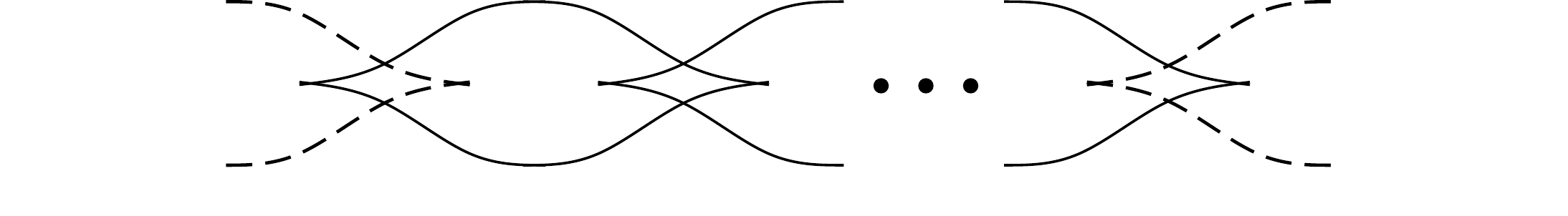
\caption{A diagram for the standard filling of $(L_i,\xi_i)$, obtained from Figure \ref{fig:filling} by deleting the 1-handle and the knot $K_i$.  Weinstein 2-handles are attached to $(B^4,\omega_{std})$ along $K_{i-1},\ldots,K_{i+1}\subset(S^3,\xi_{std})$, but not along the dashed knots $K'_i$, $K''_i$.  These are obtained from the knot $K_i$ in Figure \ref{fig:filling} by replacing the 3-balls with cusps, with the stabilizations on $K'_i$ and $K''_i$ determined by our choice of mixed torus.  Attaching a round symplectic 1-handle to this filling along $K'_i,K''_i$ yields the standard filling of $(M,\xi)$.}
\label{fig:lens-space-filling}
\end{figure}

\subsection{Distinct decompositions of fillings}
If we happen to have a classification for the fillings of the lens spaces $(L(p_1,q_1),\xi_1),\ldots,(L(p_m,q_m),\xi_m)$ produced by Theorem \ref{main-theorem}, we can then produce all of the fillings of our original hyperbolic torus bundle $(M,\xi)$ by attaching round 1-handles to these lens space fillings.  One then wonders about the extent to which this construction over-counts fillings of $(M,\xi)$.  Namely, can attaching a round 1-handle to one lens space filling yield the same filling of $(M,\xi)$ as attaching a round 1-handle to another lens space filling?\\

For the standard fillings of tight lens spaces and torus bundles, the answer to this question is yes.  From Figures \ref{fig:filling} and \ref{fig:lens-space-filling} we see that the filling of $(M,\xi)$ obtained by attaching a round 1-handle to the standard filling of $(L(p_i,q_i),\xi_i)$ is the standard filling, for $1\leq i\leq m$.

\begin{corollary}\label{cor:unique-filling}
Let $\xi$ be a virtually overtwisted tight contact structure on a hyperbolic torus bundle $M$.  If each of the tight lens spaces $(L(p_1,q_1),\xi_1),\ldots,(L(p_m,q_m),\xi_m)$ produced by Theorem \ref{main-theorem} is uniquely fillable, then so is $(M,\xi)$.
\end{corollary}

For example, consider the hyperbolic torus bundle $M$ with monodromy
\[
A = T^{-k}ST^{-m}S,
\]
for some $k,m\neq 4$, and let $\xi$ be any virtually overtwisted tight contact structure on $M$ (of which there are $(k-1)(m-1)-2$).  The precise list of tight lens spaces produced by Theorem \ref{main-theorem} will depend on $\xi$, but the only diffeomorphism types which might appear on this list are $L(k,1)$ and $L(m,1)$.  Because $k,m\neq 4$, every tight contact structure on $L(k,1)$ and $L(m,1)$ is uniquely fillable (c.f. \cite{mcduff1990structure} and \cite{plamenevskaya2010planar}).  It follows that every virtually overtwisted tight contact structure on $M$ is uniquely fillable.\\

So the hyperbolic torus bundle $(M,\xi)$ can admit non-standard fillings only when at least one of the lens spaces produced by Theorem \ref{main-theorem} admits such a filling.  But there remains the possibility that non-standard fillings of distinct lens spaces can yield the same filling of $(M,\xi)$ after round 1-handle attachment.  In Section \ref{sec:overcounting} we will describe a circumstance in which distinct lens space fillings yield the same filling of $(M,\xi)$.  Describing all such circumstances --- that is, identifying all relations among lens space fillings which lead to the same torus bundle fillings --- is an interesting question for future work.

\subsection{Strategy}\label{sec:strategy}
For each torus bundle $(M,\xi)$ that we consider, our classification will follow the same basic recipe.  We begin by identifying a fiber-isotopic \emph{mixed torus} $\Sigma\subset(M,\xi)$.  A mixed torus is a type of convex torus defined by Menke in \cite{menke2018jsj} which admits a neighborhood consisting of basic slices of opposite sign.  To identify this torus, we realize $M$ as the result of identifying the boundary components of $T^2\times I$ via some monodromy $A$.  In each case, the contact structure $\xi$ on $M$ lifts to a (perhaps non-unique) tight contact structure on $T^2\times I$, which we also denote by $\xi$.  We then use the fact that $(M,\xi)$ is virtually overtwisted to find basic slices of opposite sign on either side of $T^2\times\{0\}$ (perhaps after a contact isotopy), making $T^2\times\{0\}$ a mixed torus.\\

Having identified a mixed torus in $(M,\xi)$, we suppose that $(W,\omega)$ is an exact filling of $(M,\xi)$ and turn to Menke's JSJ decomposition theorem, \cite[Theorem 1.1]{menke2018jsj}.  This theorem produces $(W',\omega')$, an exact filling of a contact manifold $(M',\xi')$, and tells us how to obtain $(W,\omega)$ from $(W',\omega')$, as well as providing a relationship between $(M,\xi)$ and $(M',\xi')$.  Specifically, because our mixed torus is the fiber $T^2\times\{0\}$, the decomposition theorem allows us to write
\[
M' = S_0\cup (T^2\times I)\cup S_1
\]
for some identifications $\partial S_i\to T^2\times\{i\}$, where $S_0$ and $S_1$ are solid tori.  That is, $M'$ is a lens space.  Moreover, the decomposition theorem tells us that we can recover $M$ from $M'$ by removing the interiors of $S_0$ and $S_1$ and identifying the dividing curves and meridians of the resulting boundary components.  At the level of fillings, this corresponds to attaching a round symplectic 1-handle to a filling of the lens space along the cores of the solid tori --- for the lens space $(L(p_i,q_i),\xi_i)$ produced by Theorem \ref{main-theorem}, these core curves are the Legendrian knots $L^i_-$ and $L^i_+$.\\

The dividing curves of $\partial S_0$ and $\partial S_1$ have slopes $s_0$ and $s_1$, respectively, determined by the contact structure on $T^2\times I$.  If the meridians of $S_0$ and $S_1$ are denoted $\mu_0, \mu_1$, then $A\mu_1=\mu_0$.  These meridians are not \emph{a priori} determined by the decomposition theorem, but by investigating different choices for the meridian $\mu_1$, we are in many cases able to either completely determine $(M',\xi')$, or able to show that no exact filling of $(M,\xi)$ exists.\\

Our primary technique for ruling out possible meridians $\mu_1$ involves an analysis of the slopes of $\mu_0$, $\mu_1$, $\Gamma_0$, and $\Gamma_1$, where $\Gamma_i$ is the dividing set of $T^2\times\{i\}$.  Because $(M',\xi')$ is fillable, $\xi'$ must be a tight contact structure.  Now consider a family of convex tori in $M'$, beginning with the boundary of a tubular neighborhood of the core of $S_1\subset M'$, passing through the fibers of $T^2\times I\subset M'$, and tending towards the boundary of a tubular neighborhood of the core of $S_0\subset M'$.  Each torus in this family has a pair of dividing curves, and the slopes of these curves will vary from the slope of $\mu_1$ to $s_1$ to $s_0$ to the slope of $\mu_0$ as we traverse the family.  Because $\xi'$ is tight, the total angle through which these dividing curves pass cannot exceed $\pi$.  We will rule out many possibilities for the meridian $\mu_1$ by showing that these choices would cause this last condition to be violated.

\subsection*{Acknowledgements}
The author is grateful to Ko Honda for suggesting this project and offering feedback throughout its completion, and would also like to thank Youlin Li and Emmy Murphy for a helpful correspondence and conversation, respectively.  Thanks also to an anonymous referee for several helpful suggestions.

\section{Background}\label{sec:background}
In this section we recall some definitions necessary for the statement of Theorem \ref{main-theorem} --- such as \emph{symplectic filling} and \emph{virtually overtwisted torus bundle} --- and offer a brief review of \cite[Theorem 1.1]{menke2018jsj}, which is the main ingredient of our proof.

\subsection{Symplectic fillings of contact manifolds}
We will say that a contact manifold $(M,\xi)$ is \emph{fillable} if it can be realized as the boundary of a symplectic manifold, with a certain level of compatibility required between the symplectic and contact structures.  More precisely, we have the following definitions.

\begin{definition}
Fix a co-oriented contact manifold $(M,\xi)$ and suppose $(W,\omega)$ is a symplectic manifold with $\partial W=M$ as oriented manifolds.  We say that $(W,\omega)$ is
\begin{itemize}
	\item a \emph{strong symplectic filling} of $(M,\xi)$ if there is a 1-form $\lambda$ on $W$ such that $\omega=d\lambda$ on some neighborhood of $\partial W$ and $\xi=\ker(\lambda|_{\partial W})$;
	\item an \emph{exact filling} of $(M,\xi)$ if there is a 1-form $\lambda$ on $W$ such that $\omega=d\lambda$ on all of $W$ and $\xi=\ker(\lambda|_{\partial W})$.
\end{itemize}
We say that $(M,\xi)$ is \emph{strongly symplectically fillable} or \emph{exactly fillable} if it admits a strong symplectic or exact fillling, respectively.
\end{definition}

\subsection{Tight contact structures on torus bundles}
The contact manifolds of central interest in this paper are \emph{torus bundles over $S^1$}.  These are smooth 3-manifolds whose tight contact structures were classified by Honda in \cite{honda2000classification2}.\\

In \cite{honda2000classification2} Honda gave a classification of the tight contact structures on $M_A$ in terms of the monodromy $A$, and each of these structures is either \emph{universally tight} or \emph{virtually overtwisted}, notions we now define.

\begin{definition}
Let $(M,\xi)$ be a tight contact 3-manifold, $\tilde{\pi}\colon\tilde{M}\to M$ the universal cover of $M$.  We will say that $\xi$ is \emph{universally tight} if the contact manifold
\[
(\tilde{M},\tilde{\xi}:=\tilde{\pi}^*\xi)
\]
is tight.  We will call $\xi$ \emph{virtually overtwisted} if there is some finite cover $\overline{\pi}\colon\overline{M}\to M$ of $M$ for which the contact manifold
\[
(\overline{M},\overline{\xi}:=\overline{\pi}^*\xi)
\]
is overtwisted.
\end{definition}

\begin{remark}
It is not immediately obvious that every tight structure on a 3-manifold must be either universally tight or virtually overtwisted.  However, every tight contact structure on a manifold with residually finite fundamental group must fall into one of these categories, and work of Hempel (\cite{hempel1987residual}) along with the geometrization conjecture shows that every 3-manifold has this property.
\end{remark}

We are now prepared to summarize Honda's classification of virtually overtwisted tight contact structures on torus bundles over $S^1$.  This classification is contained in the following table, omitting torus bundles which admit only universally tight contact structures.  We continue to use the generators $S$ and $T$ of $SL(2,\mathbb{Z})$ identified in Section \ref{sec:intro-torus-bundles}.
\begin{center}
\begin{tabular}{|c|c|c|}
\hline
{\bf Type} & {\bf Monodromy} & {\bf \# of VOT structures} \\\hline
\multirow{2}{*}{Elliptic ($|\tr~A|<2$)} & $A=S$ & 1\\
& $A=(T^{-1}S)^2$ & 2\\ \hline
\multirow{4}{*}{Parabolic ($|\tr~A|=2$)} & $A=T^2$ & 1\\
& $A=T^n$, $n>2$ & 2\\
& $A=T^n$, $n\leq -2$ & $|n-1|-2$\\
& $A=-T^n$, $n<0$ & 1\\ \hline
Hyperbolic ($|\tr~A|>2$) & $A=T^{r_0}ST^{r_1}S\cdots T^{r_k}S$ & $|(r_0+1)\cdots(r_k+1)|-2$\\
($r_0\leq -3,r_i\leq -2$) & $A=-T^{r_0}ST^{r_1}S\cdots T^{r_k}S$ & $|(r_0+1)\cdots(r_k+1)|$\\ \hline
\end{tabular}
\end{center}\vspace{0.5em}

We will not say much about Honda's techniques for establishing the above classification, but will recall some of the vocabulary used in the proof.\\

\begin{definition}
Let $(M,\xi)$ be a contact 3-manifold, and let $\Sigma\subset M$ be a surface.  We will say that $\Sigma$ is \emph{convex} if there exists a contact vector field $v$ which is transverse to $\Sigma$.  That is, if $\lambda$ is a 1-form on $M$ satisfying $\ker\lambda=\xi$, then $\ker(\mathcal{L}_v\lambda)=\xi$.  If $\Sigma$ is convex, we define the \emph{dividing set} of $\Sigma$ to be
\[
\Gamma_\Sigma = \{p\in\Sigma|v_p\in\xi_p\}.
\]
The dividing set is a multicurve which depends on $v$ only up to isotopy.
\end{definition}

Throughout this paper, whenever we have a convex torus $T^2\subset(M,\xi)$, we will assume that the dividing set $\Gamma_{T^2}$ has two parallel components.  Under the identification $T^2=\mathbb{R}^2/\mathbb{Z}^2$, these components are isotopic to a line of rational slope $s(T^2)$, which we refer to as the slope of the dividing set.  We may now define the fundamental building block for classifying tight structures on torus bundles --- the \emph{basic slice}.\\

\begin{figure}
\centering
\input{farey.tex}
\caption{The Farey tessellation of the hyperbolic unit disk.  The rational numbers $p/q$ and $p'/q'$ are each connected to $(p+p')/(q+q')$ by an arc.}
\label{fig:farey}
\end{figure}
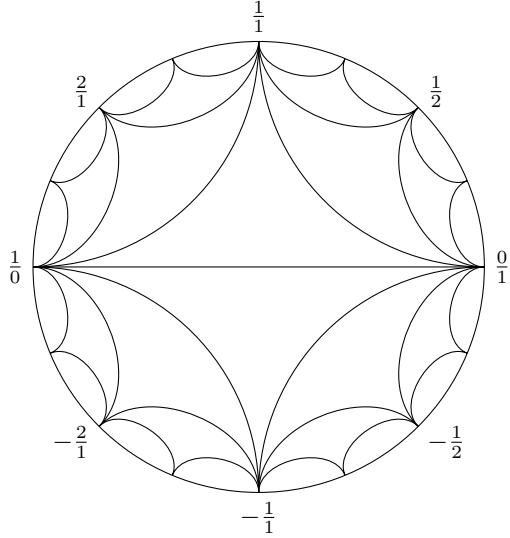

\begin{definition}
Let $(T^2\times I,\xi)$ be tight, with convex boundary, and let $s_i$ be the slope of $T^2\times\{i\}$, $i=0,1$.  We call $(T^2\times I,\xi)$ a \emph{basic slice} if
\begin{itemize}
	\item $s_0$ and $s_1$ are connected by an edge of the Farey tessellation;
	\item for each $t\in[0,1]$, if $T^2\times\{t\}$ is convex, then its slope $s$ lies in the interval $[s_1,s_0]$ on the Farey tessellation.
\end{itemize}
The Farey tessellation is depicted in Figure \ref{fig:farey}, with $[s_1,s_0]$ denoting a counterclockwise arc from $s_1$ to $s_0$.
\end{definition}

\begin{remark}
The rational numbers $s_0$ and $s_1$ will be connected by an edge of the Farey tessellation if and only if the minimal integral vectors representing $s_0,s_1$ form a basis of $\mathbb{Z}^2$.
\end{remark}

For our purposes, the most pertinent fact about basic slices is the following: once we fix rational slopes $s_0,s_1$ which are Farey neighbors, there are precisely two tight contact structures $\xi$ on $T^2\times I$ making $(T^2\times I,\xi)$ a basic slice with slopes $s_0,s_1$, and these structures are distinguished by their relative Euler classes.  Indeed, we have $PD(e(\xi,s))=\pm(0,1)\in H_1(T^2;\mathbb{Z})$.\\

Basic slices are useful in the classification of tight contact structures because one can decompose a tight solid torus $(T^2\times I,\xi)$ into basic slices, each of which admits two tight contact structures.  These basic slices naturally cluster into \emph{continued fraction blocks}, within which the basic slices may be shuffled.  The standard model for a continued fraction block is a tight structure on $T^2\times [0,m]$ with boundary slopes $s_m=-1-m$, $s_0=-1$ for some $m\geq 1$.  We assume that this tight structure decomposes into $m$ basic slices, with $s_k=-1-k$ for $0\leq k\leq m$.  There are $m+1$ tight structures satisfying these requirements, distinguished by their relative Euler classes, which satisfy $PD(e(\xi,s))=(0,k)\in H_1(T^2;\mathbb{Z})$ for some $k\in\{-m,2-m,\ldots,m-2,m\}$.\\

Of the $m+1$ tight contact structures on a continued fraction block $T^2\times[0,m]$, just two --- those with $PD(e(\xi,s))=\pm(0,m)$ --- are universally tight.  The remaining tight structures will contain adjacent basic slices of opposite sign, which causes their lifts to the universal cover to be overtwisted.  More generally, a tight contact structure $\xi$ on a thickened torus $T^2\times I$ can have each of its continued fraction blocks be universally tight, while $\xi$ itself is virtually overtwisted.  Again, this occurs when the basic slice decomposition of $(T^2\times I,\xi)$ includes adjacent basic slices of opposite sign.  See \cite[Section 4.4.5]{honda2000classification2} for details.\\

Following Honda, we will think of our contact manifolds $(M_A,\xi)$ as thickened tori $(T^2\times I,\xi)$, with the ends $T^2\times\partial I$ identified by the monodromy $A$.  When the monodromy is negative (in the sense described above), this identification will induce a change of sign, leading some universally tight contact structures on $T^2\times I$ to induce virtually overtwisted contact structures on $M_A$.  Again, see \cite{honda2000classification2} for full details.

\subsection{Mixed tori and the JSJ decomposition}\label{sec:background-jsj}
The purpose of this subsection is to state \cite[Theorem 1.1]{menke2018jsj}, which is our primary tool in proving Theorem \ref{main-theorem}.  When a contact manifold $(M,\xi)$ admits a certain type of convex torus $T^2\subset(M,\xi)$, \cite[Theorem 1.1]{menke2018jsj} allows us to cut open a filling $(W,\omega)$ of $(M,\xi)$ along $T^2$.  This yields a new filling $(W',\omega')$ of a new contact manifold $(M',\xi')$, and our strategy is to leverage an understanding of the fillings of $(M',\xi')$ into information about the fillings of $(M,\xi)$.\\

The restriction Menke places on the convex torus $T^2\subset (M,\xi)$ is that it must be \emph{mixed}.  We call $T^2$ a \emph{mixed torus} if there is a neighborhood $T^2\times[-1,1]\subset M$ of $T^2=T^2\times\{0\}$ such that
\begin{itemize}
	\item $(T^2\times[-1,0],\xi_{T^2\times[-1,0]})$ and $(T^2\times[0,1],\xi_{T^2\times[0,1]})$ are basic slices;
	\item $(T^2\times[-1,1],\xi_{T^2\times[-1,1]})$ is virtually overtwisted.
\end{itemize}
This will occur precisely when $PD(e(\xi_{T^2\times[-1,1]},s))=(0,0)\in H_1(T^2;\mathbb{Z})$.  The mixed torus terminology arises from thinking of $T^2$ as sandwiched between basic slices of opposite sign.\\

The symplectic filling $(W',\omega')$ produced by Menke's result is related to the original filling $(W,\omega)$ by \emph{round symplectic 1-handle attachment}, a notion described in \cite{avdek2012liouville} and \cite{adachi2017round}.  Briefly, a round symplectic 1-handle is described as follows.  Let $(S,\xi)$ be a standard contact torus.  With coordinates $(\theta,x,y)$ on $S=S^1\times\mathbb{D}^2$, we may write $\xi=\ker(dy-xd\theta)$.  We then let $H=[-1,1]\times S$ and set
\[
\omega_H = dt\wedge dy + d\theta\wedge dx,
\]
where $t$ is the coordinate on $[-1,1]$, and call $(H,\omega_H)$ a \emph{round symplectic 1-handle}.  We may define a vector field $V$ on $H$ so that $\mathcal{L}_V\omega_H=\omega_H$ and $V$ points transversely out of $H$ along $[-1,1]\times \partial S\subset\partial H$, while pointing transversely into $H$ along $\{\pm 1\}\times S\subset\partial H$.  This allows us to attach $(H,\omega_H)$ to a strong symplectic filling $(W,\omega)$ to obtain a new such filling.  One does this by identifying standard solid tori $S_0,S_1\subset (M,\xi)=\partial W$, which have natural identifications with $\{-1\}\times S$ and $\{1\}\times S$, respectively.  Typically the solid tori $S_0,S_1$ arise as standard neighborhoods of Legendrian knots in $(M,\xi)$.  For full details, see \cite{avdek2012liouville} or \cite{adachi2017round}.\\

We can now state Menke's result.
\begin{theorem}[{\cite[Theorem 1.1]{menke2018jsj}}]\label{thm:jsj}
Let $(M,\xi)$ be a closed, cooriented contact 3-manifold, and let $(W,\omega)$ be a strong (exact) symplectic filling of $(M,\xi)$.  If there exists a mixed torus $T^2\subset(M,\xi)$, then there exists a symplectic manifold $(W',\omega')$ such that
\begin{enumerate}
	\item $(W',\omega')$ strongly (exactly) fills its boundary $(M',\xi')$;
	\item there are standard solid tori $S_0,S_1\subset(M',\xi')$ for which
	\[
	M = (M' - \mathrm{int}(S_0) - \mathrm{int}(S_1))/\sim,
	\]
	where $\sim$ identifies the meridians and dividing sets of $\partial S_0$ and $\partial S_1$;
	\item the filling $(W,\omega)$ can be recovered from $(W',\omega')$ by attaching a round symplectic 1-handle along $S_0$ and $S_1$.
\end{enumerate}
\end{theorem}

\section{Proof of Theorem \ref{main-theorem}}
In this section we use Theorem \ref{thm:jsj} to prove Theorem \ref{main-theorem}.  We treat the elliptic, parabolic, and hyperbolic cases separately.

\subsection{Elliptic torus bundles}
There are two elliptic torus bundles which admit a virtually overtwisted contact structure --- one with monodromy $A=S$, and the other with monodromy $A=(T^{-1}S)^2$.  In both cases we obtain a virtually overtwisted structure by starting with a minimally twisting tight structure on $T^2\times I$ with boundary slopes $s_1=0$ and $s_0=\infty$ and then passing to the torus bundle $M_A$.  There are two such structures on $T^2\times I$, and they become contact isotopic on $M_A$ when $A=S$.  When $A=(T^{-1}S)^2$, the structures remain distinct.\\

So there are three virtually overtwisted, elliptic torus bundles.  None admit a strong symplectic filling.

\begin{proposition}
Let $(M,\xi)$ be a virtually overtwisted, elliptic torus bundle.  Then $(M,\xi)$ is not strongly symplectically fillable.
\end{proposition}
\begin{proof}
As stated above, there are three contact manifolds $(M,\xi)$ which satisfy the hypotheses of this proposition, and all three are obtained from a tight $(T^2\times I,\xi)$ with boundary slopes $s_1=0$, $s_0=\infty$.  We consider the case $A=S$; the case $A=(T^{-1}S)^2$ is similar.\\

Our first claim is that the image of the fiber $T^2\times\{0\}$ in $(M,\xi)$ is a mixed torus, and our proof of this fact mimics Honda's proof of the fact that $\xi$ is an overtwisted contact structure. (c.f. \cite[Section 4]{honda2000classification}).  In particular, we begin with a tight contact structure $\xi$ on $T^2\times[0,1]$ with boundary slopes $s_1=0$ and $s_0=\infty$.  Because these slopes are Farey neighbors, $(T^2\times[0,1],\xi)$ is a basic slice, and there are precisely two possibilities for $\xi$, distinguished by their relative Euler classes.  As in \cite[Section 4.7]{honda2000classification}, we compute these relative Euler classes to be
\[
PD(e(\xi,s)) = \pm(v_1-v_0) = \pm((1,0)-(0,1)) = \pm(1,-1) \in H_1(T^2;\mathbb{Z}),
\]
where $v_0$ and $v_1$ are minimal vectors in $\mathbb{Z}^2$ representing the slopes $s_0$ and $s_1$, respectively.  Now we may decompose $T^2\times[0,1]$ into a pair of basic slices
\[
(N_1:=T^2\times[0,1/2],\xi_1:=\xi|_{N_1})
\quad\text{and}\quad
(N_2:=T^2\times[1/2,1],\xi_2:=\xi|_{N_2})
\]
with boundary slope $s_{1/2}=1$, since $1$ is a Farey neighbor of both $0$ and $\infty$.  The relative Euler classes of these basic slices are then
\[
PD(e(\xi_1,s))=\pm(v_{1/2}-v_0)
\quad\text{and}\quad
PD(e(\xi_2,s))=\pm(v_1-v_{1/2}),
\]
where $v_{1/2}=(1,1)$.  Because these basic slices glue to form the basic slice $(T^2\times[0,1],\xi)$, their signs agree.\\

Now because the ends of $T^2\times[0,1]$ are identified to form $M$, we may also view $N_2$ as $T^2\times[-1/2,0]$ by applying the monodromy $A$.  Applying the monodromy alters the relative Euler class:
\[
PD(e(\xi_2,s)) = \pm A(v_1-v_{1/2}) = \pm(Av_1-Av_{1/2}) = \pm(-v_0-(v_1-v_0))=\mp v_1.
\]
Here we are using the fact that $A=S$ is a rotation by $-\pi/2$.  So the basic slice $(T^2\times[-1/2,0],\xi_{T^2\times[-1/2,0]})$ has a sign which is necessarily opposite that of the basic slice $(T^2\times[0,1/2],\xi_{T^2\times[0,1/2]})$, making the fiber $T^2\times\{0\}$ sandwiched between them a mixed torus.\\

In \cite{honda2000classification}, Honda uses this relative Euler class computation to show that $(M,\xi)$ is virtually overtwisted.  An analogous computation for the monodromy $A=(T^{-1}S)^2$ allows us to distinguish between the two virtually overtwisted tight contact structures on this torus bundle, and to see that the fiber $T^2\times\{0\}$ is a mixed torus in each of them.\\

As outlined in our strategy, we now suppose that $(M,\xi)$ is fillable, with strong symplectic filling $(W,\omega)$, using the notation from Section \ref{sec:strategy}.  Menke's decomposition theorem, applied to $(W,\omega)$ and $T^2\times\{0\}\subset(M,\xi)$, produces a strong filling $(W',\omega')$ of $(M',\xi')$, with
\[
M'\simeq S_0 \cup (T^2\times I)\cup S_1.
\]
Because $s_1=0$ and $s_0=\infty$, the dividing sets $\Gamma_i$ of $\partial S_i$ can be represented by
\[
\Gamma_0 = \left(\begin{matrix} 0\\ 1\end{matrix}\right),
\quad
\Gamma_1 = \left(\begin{matrix} 1\\ 0\end{matrix}\right).
\]
The shortest integer vector representing the meridian $\mu_1$ must form an integral basis with $\Gamma_1$ for $\mathbb{Z}^2$, and we have $\mu_0=A\mu_1$.  So we have representatives
\[
\mu_1 = \left(\begin{matrix} m\\ 1\end{matrix}\right)
\quad\text{and}\quad
\mu_0 = \left(\begin{matrix} 1\\ -m\end{matrix}\right)
\]
for some $m\in\mathbb{Z}$.\\

\begin{figure}
\centering
\begin{subfigure}{0.45\linewidth}
\centering
\input{1a.tex}
\caption{$A=S$.}
\label{fig:1a}
\end{subfigure}
\begin{subfigure}{0.45\linewidth}
\centering
\input{1b.tex}
\caption{$A=(T^{-1}S)^2$.}
\label{fig:1b}
\end{subfigure}
\caption{Slope analysis for virtually overtwisted, elliptic torus bundles.}
\label{fig:1}
\end{figure}
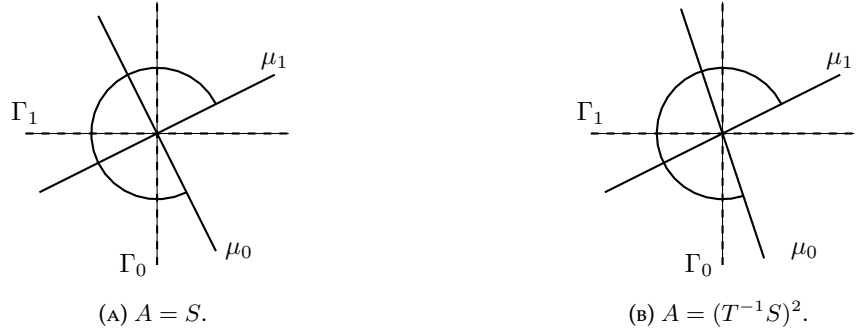

Notice that if $m\geq 0$, then a counterclockwise rotation from $\mu_1$ to $\Gamma_1$ passes through an angle of at least $\pi/2$, and a counterclockwise rotation from $\Gamma_1$ to $\Gamma_0$ takes us through an angle of precisely $\pi/2$.  The rotation from $\Gamma_0$ to $\mu_0$ is then non-trivial, and we see that our dividing curves pass through an angle in excess of $\pi$ as we move from the core of $S_1$ to $T^2\times I$, and on to the core of $S_0$.  The same rotation occurs when $m\leq 0$.  In fact, this rotation is always $3\pi/2$ --- regardless of the value of $m$ --- as can be seen in Figure \ref{fig:1a}.\\

The case $A=(T^{-1}S)^2$ is not much different.  The dividing curves will rotate through an angle of $3\pi/2-\theta(m)$, where $\theta(m)$ is the acute angle between $(1,m)^T$ and $(1,-m-1)^T$.  In particular, the dividing curves rotate through an angle greater than $\pi$.  The case $m=2$ is depicted in Figure \ref{fig:1b}.\\

In either case, we see that $(M',\xi')$ contains a thickened torus with non-minimal twisting which may be completed to a solid torus $S'$ strictly containing $S_1$.  The non-minimal twisting means that $(S',\xi|_{S'})$ is overtwisted (c.f. \cite[Section 2.3]{honda2000classification}), and thus so is $(M',\xi')$.  Note that this is the case no matter the value of $m$.  But of course this contradicts the fillability of $(M',\xi')$, so we conclude that no strong symplectic filling $(W,\omega)$ of $(M,\xi)$ exists.
\end{proof}

\subsection{Parabolic torus bundles}
Topologically, the parabolic torus bundles admitting virtually overtwisted contact structures are all circle bundles, with base either the torus or the Klein bottle.  The virtually overtwisted structures on these manifolds then fall into three families.\\

For $n\geq 2$, the monodromy $A=T^n$ produces a circle bundle over $T^2$ with Euler number $n$.  This bundle admits a unique virtually overtwisted contact structure if $n=2$, and admits two such structures if $n>2$.  These structures are exceptional in that they are the only virtually overtwisted contact structures on any torus bundles which fail to be \emph{minimally twisting}, a notion defined by Honda in \cite{honda2000classification2}.  We will show that these structures are not strongly symplectically fillable.

\begin{proposition}
Fix $n\geq 2$ and let $M_n$ be the torus bundle with monodromy $A=T^n$.  Let $\xi$ be a virtually overtwisted contact structure on $M$.  Then $(M_n,\xi)$ is not strongly symplectically fillable.
\end{proposition}
\begin{proof}
Per the classification in \cite{honda2000classification}, there are two tight contact structures on $T^2\times I$ with boundary slopes $s_0=s_1=0$, and these pass to tight contact structures on $M_n$.  In \cite{honda2000classification2}, Honda shows that these structures are distinct when $n>2$, and for $n\geq 2$ give all virtually overtwisted contact structures on $M_n$.  Let $\xi$ be one of these structures on $T^2\times I$, passing to $\xi$ (via notational abuse) on $M$.\\

Because $(T^2\times I,\xi)$ is not minimally twisting, the dividing curves rotate through an angle of $\pi$ as we move from $T^2\times\{0\}$ to $T^2\times\{1\}$, and thus all slopes are achieved by some boundary-parallel torus.  We take $T^2\times\{1/2\}$ with slope $s_{1/2}=\infty$.  Because $\infty$ is connected to $0$ by an edge of the Farey tessellation, each of $T^2\times[0,1/2]$ and $T^2\times[1/2,1]$ is a basic slice, and the fact that their union $T^2\times I$ is not universally tight makes $T^2\times\{1/2\}$ a mixed torus.\\

Now suppose that $(M_n,\xi)$ admits a strong symplectic filling $(W,\omega)$, and let $(M',\xi')$ be the strongly symplectically fillable manifold that results from applying Menke's decomposition theorem to $(W,\omega)$, using the mixed torus $T^2\times\{1/2\}$.  In order to split $(M_n,\xi)$ open along $T^2\times\{1/2\}$, we think of this torus bundle as the result of identifying the ends of $(T^2\times[-1/2,1/2],\xi)$ via the monodromy.  Then, as above, $M'$ results from gluing solid tori $S_{-1/2}$ and $S_{1/2}$ to the boundary components of $T^2\times[-1/2,1/2]$, and we are left to determine the meridians $\mu_{-1/2},\mu_{1/2}$.  Notice that the dividing curves $\Gamma_{-1/2}$ and $\Gamma_{1/2}$ have slopes $s_{-1/2}=1/n$ and $s_{1/2}=\infty$, respectively.  Because the shortest integer vectors representing $\Gamma_{1/2}$ and $\mu_{1/2}$ must form an integer basis for $\mathbb{R}^2/\mathbb{Z}^2$, we may represent $\mu_{1/2}$ by the vector $(1,m)^T$, for some $m\in\mathbb{Z}$.\\

Suppose we have $m=0$, so that $\mu_{1/2}=(1,0)^T$.  Then $\mu_{-1/2}=A(1,0)^T=(1,0)$, so the counterclockwise rotation from $\mu_{1/2}$ to $\Gamma_{1/2}$ to $\Gamma_{-1/2}$ to $\mu_{-1/2}$ takes us through an angle of $2\pi$, and we find that $(M',\xi')$ is overtwisted.  So $m\neq 0$.\\

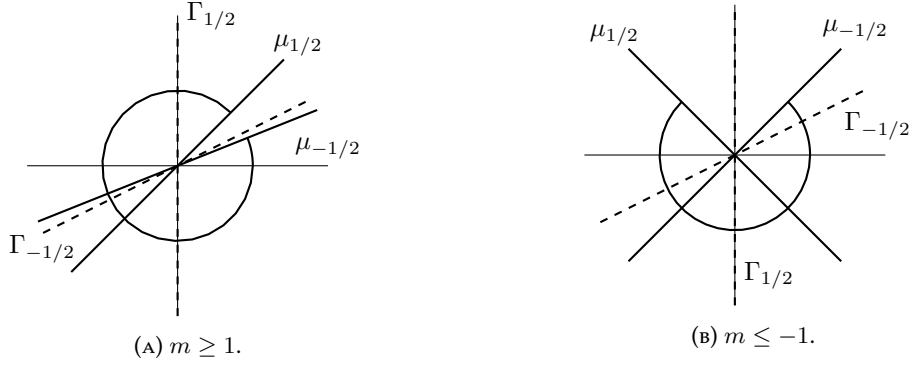
\begin{figure}
\centering
\begin{subfigure}{0.45\linewidth}
\centering
\input{2a-m-pos.tex}
\caption{$m\geq 1$.}
\end{subfigure}
\begin{subfigure}{0.45\linewidth}
\centering
\input{2a-m-neg.tex}
\caption{$m\leq -1$.}
\end{subfigure}
\caption{Slope analysis for a virtually overtwisted torus bundle with monodromy $T^n$, $n\geq 2$.}
\label{fig:2a}
\end{figure}

Next, suppose $m\geq 1$, meaning that $\mu_{-1/2}$ is represented by $A(1,m)^T=(1+mn,m)$.  Because $n\geq 2$, we have
\[
0 < \frac{m}{1+mn} < \frac{1}{n} < m < \infty,
\]
which is to say that
\[
0 < \slope(\mu_{-1/2}) < \slope(\Gamma_{-1/2}) < \slope(\mu_{1/2}) < \slope(\Gamma_{1/2}).
\]
In this case, the dividing curves rotate through an angle in excess of $3\pi/2$, and again $(M',\xi')$ is overtwisted.  So $m<0$.\\

But if $m<0$ then we find that
\[
\slope(\mu_{1/2}) < 0 < \slope(\Gamma_{-1/2}) < \slope(\mu_{-1/2}) < \slope(\Gamma_{1/2}),
\]
since $m<0<1/n<m/(1+mn)$.  In this case, counterclockwise rotation from $\mu_{1/2}$ to $\Gamma_{1/2}$ to $\Gamma_{-1/2}$ to $\mu_{-1/2}$ takes us through an angle of $3\pi/2$, for any choice of $m<0$ and $n\geq 2$, and again we find that $(M',\xi')$ is overtwisted.  Because there are no suitable choices for the meridian $\mu_{1/2}$, we must conclude that $(M_n,\xi)$ does not admit a strong filling $(W,\omega)$.  The overtwistedness when $m\neq 0$ is seen in Figure \ref{fig:2a}.
\end{proof}

When $n\leq -2$, the monodromy $A=T^n$ will again give us a circle bundle over $T^2$ with Euler number $n$, but in this case we have $|n-1|-2$ virtually overtwisted structures.  These bundles admit a unique strong filling, up to symplectic deformation equivalence and blowup, for any virtually overtwisted structure.  As a corollary, the exact fillings of these bundles are unique up to symplectomorphism.

\begin{proposition}
Fix $n\leq -2$ and let $M_n$ be the torus bundle with monodromy $A=T^n$.  Let $\xi$ be a virtually overtwisted contact structure on $M_n$.  Then $(M_n,\xi)$ admits a unique strong filling, up to symplectic deformation equivalence and blowup.
\end{proposition}
\begin{proof}
Consider a tight contact structure $\xi'$ on $T^2\times I$ with boundary slopes $s_0=\frac{-1}{1-n}$ and $s_1=-1$.  Honda showed in \cite{honda2000classification} that there are $1-n$ such structures, distinguished by the number $k$ of positive basic slices they have in a continued fraction block.  Concretely, each such structure $\xi'$ admits a decomposition into $-n$ basic slices
\begin{equation}\label{eq:1-b-basic-slice-decomposition}
\left(T^2\times\left[0,\frac{1}{n}\right]\right)\cup
\left(T^2\times\left[\frac{1}{n},\frac{2}{n}\right]\right)\cup
\cdots\cup
\left(T^2\times\left[\frac{n-1}{n},1\right]\right),
\end{equation}
and we can identify $\xi'$ by counting the number of these basic slices which are positive.  These structures remain tight and distinct when we pass to $M_n$, and if $k$ is neither 0 nor $1-n$, the resulting structure is virtually overtwisted.  If $k=0$ or $k=1-n$, the basic slices above all have the same sign, and the structure on $M_n$ is universally tight.\\

Let $\xi$ be a virtually overtwisted structure on $M_n$, induced by a structure $\xi'$ on $T^2\times I$.  Because $\xi$ is virtually overtwisted, the basic slices in the decomposition (\ref{eq:1-b-basic-slice-decomposition}) do not all have the same sign.  In particular, we may shuffle the basic slices so that that the first and last basic slices have opposite sign, ensuring that the image of $T^2\times\{0\}$ in $(M_n,\xi)$ is a mixed torus.\\

Now suppose that $(W,\omega)$ is an exact filling of $(M_n,\xi)$ and let $(W',\omega')$ be the symplectic manifold with boundary produced by splitting $(W,\omega)$ open along our mixed torus.  This manifold exactly fills its boundary $(M',\xi')$, and we may write
\[
M' = S_0 \cup (T^2\times I) \cup S_1,
\]
where the boundary slopes of $T^2\times I$ are $s_0=\frac{-1}{1-n}$ and $s_1=-1$.  Denote by $\Gamma_i$ and $\mu_i$ the dividing curves and meridian, respectively, of $\partial S_i$, for $i=0,1$.  Then $\Gamma_0$ and $\Gamma_1$ are represented by the vectors $(n-1,1)^T$ and $(1,-1)^T$, respectively.  Because the shortest integer vectors representing $\mu_1$ and $\Gamma_1$ must form an integral basis for $\mathbb{Z}^2$, we may represent $\mu_1$ by $(m,k)^T$, with $m\geq 0$ and $k=\pm 1-m$.  It follows that $\mu_0$ is represented by
\[
A\left(\begin{matrix} m\\ \pm 1-m \end{matrix}\right)
=\left(\begin{matrix}
\pm n - m(n-1)\\ \pm 1-m
\end{matrix}\right).
\]
We now begin ruling out candidate values for $m$.\\

First, $m$ must be positive.  Indeed, if $m=0$ then the slopes of $\mu_1$ and $\mu_0$ are given by $\infty$ and $1/n$, respectively.  It follows that as we move a convex torus from the core of $S_1$ to $T^2\times\{1\}$, then to $T^2\times\{0\}$ and on to the core of $S_0$, the dividing curve slopes will range from $\infty$ to $-1$ to $\frac{-1}{1-n}$ and on to $1/n$, rotating through an angle greater than $\pi$, as indicated in Figure \ref{fig:2c-m-0}.  This would mean that $(M',\xi')$ is overtwisted, violating its fillability.\\

Next, we must have $k=1-m$.  To this end, notice that for $n\leq -2$ and $m\geq 1$ we have
\[
\dfrac{-1-m}{m} < -1 < \dfrac{-1-m}{-n-m(n-1)} < \dfrac{-1}{1-n}.
\]
This tells us that if $k=-1-m$, then
\[
\slope(\mu_1) < \slope(\Gamma_1) < \slope(\mu_0) < \slope(\Gamma_0),
\]
and again the dividing curves pass through too great an angle for $(M',\xi')$ to be tight.  This rotation is seen in Figure \ref{fig:2c-k-minus}.  So $k=1-m$ and $\mu_1$ is represented by a vector of the form $(m,1-m)^T$, with $m\geq 1$.\\

Finally we show that $m=1$.  If $m\geq 2$ then we have
\[
-1 < \dfrac{1-m}{m} < \dfrac{-1}{1-n} < \dfrac{1-m}{n-m(n-1)},
\]
so
\[
\slope(\Gamma_1) < \slope(\mu_1) < \slope(\Gamma_0) < \slope(\mu_0),
\]
and yet again the counterclockwise rotation from $\mu_1$ to $\Gamma_1$ to $\Gamma_0$ to $\mu_0$ takes us through an angle in excess of $\pi$.  This rotation can be seen in Figure \ref{fig:2c-geq-2}.\\

So the meridians are given by $\mu_1=(1,0)^T$ and $\mu_0=(1,0)^T$, meaning that $M'\simeq S^2\times S^1$.  This determines $\xi'$, since $S^2\times S^1$ has a unique tight contact structure $\xi_0$, and in fact determines $(W',\omega')$, since $(S^2\times S^1,\xi_0)$ has a unique strong filling up to symplectic deformation equivalence and blowup.  Because we have a recipe for reconstructing $(W,\omega)$ from $(W',\omega')$ --- namely, by attaching a round symplectic 1-handle to $(W',\omega')$ along $S_0$ and $S_1$ --- we see that this filling is uniquely determined by $(M_n,\xi)$.
\end{proof}

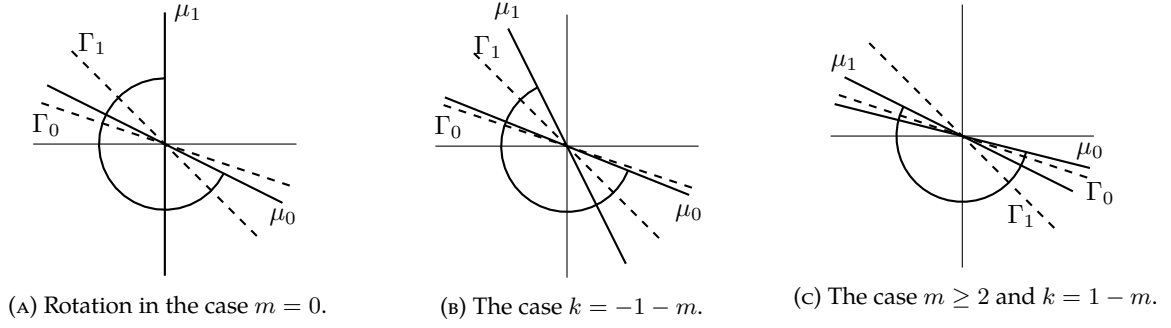
\begin{figure}
\centering
\begin{subfigure}{0.32\linewidth}
\centering
\input{2c-m-0.tex}
\caption{Rotation in the case $m=0$.}
\label{fig:2c-m-0}
\end{subfigure}
\begin{subfigure}{0.32\linewidth}
\centering
\input{2c-k-minus.tex}
\caption{The case $k=-1-m$.}
\label{fig:2c-k-minus}
\end{subfigure}
\begin{subfigure}{0.32\linewidth}
\centering
\input{2c-geq-2.tex}
\caption{The case $m\geq 2$ and $k=1-m$.}
\label{fig:2c-geq-2}
\end{subfigure}
\caption{Slope analysis for the monodromy $A=-T^n$, $n\leq -1$ or $A=T^n,n\leq -2$.}
\end{figure}

Finally, for $n\leq -1$ we consider the monodromy $A=-T^n$, which produces a non-orientable circle bundle over the Klein bottle.  This bundle admits a unique virtually overtwisted contact structure, and in \cite{ding2018strong}, Ding and Li constructed a Stein filling for this structure.  We show that Ding-Li's filling is the unique exact filling of this torus bundle, up to symplectomorphism.  Indeed, theirs is the only strong filling, up to symplectic deformation equivalence and blowup.

\begin{proposition}
Fix $n\leq -1$ and let $(M_n,\xi)$ be the virtually overtwisted torus bundle with monodromy $A=-T^n$.  Then $(M_n,\xi)$ admits a unique exact filling, up to symplectic deformation equivalence and blowup.
\end{proposition}
\begin{proof}
As mentioned, Ding and Li construct a Stein filling of $(M_n,\xi)$ in \cite{ding2018strong}.  We use the mixed torus approach to show that no other exact fillings exist.  Honda showed in \cite{honda2000classification2} that the $1-n$ distinct tight contact structures on $T^2\times I$ with boundary slopes $s_0=\frac{-1}{1-n}$ and $s_1=-1$ all descend to $\xi$ on $M_n$.  Each of these structures on $T^2\times I$ is divided into $-n$ basic slices
\[
\left(T^2\times\left[0,\frac{1}{n}\right]\right)\cup
\left(T^2\times\left[\frac{1}{n},\frac{2}{n}\right]\right)\cup
\cdots\cup
\left(T^2\times\left[\frac{n-1}{n},1\right]\right),
\]
and the contact structure on $T^2\times I$ is determined by the number of positive basic slices in this decomposition.  If $n\leq -2$, then there are $-n\geq 2$ basic slices, and we can choose a structure on $T^2\times I$ for which the basic slices $T^2\times[0,\frac{1}{n}]$ and $T^2\times[\frac{n-1}{n},1]$ are positive.  The change of sign induced by $A$ will then cause the image of $T^2\times\{0\}$ in $(M_n,\xi)$ to sit between basic slices of opposite sign --- that is, $T^2\times\{0\}$ will be a mixed torus.\\

In case $n=-1$, either of the contact structures on $T^2\times I$ with $s_1=-1$ and $s_0=-1/(1-n)=-1/2$ is a basic slice.  There is some $t_0\in I$ so that the convex torus $T^2\times\{t_0\}$ has slope $s_{t_0}=-2/3$, so we may further divide $T^2\times I$ into a pair of basic slices
\[
T^2\times I = \left(T^2\times[0,t_0]\right) \cup \left(T^2\times[t_0,1]\right),
\]
each with the same sign.  As before, the change of sign produced by the monodromy allows us to realize $T^2\times\{0\}$ as a mixed torus, sitting between the basic slices $T^2\times[0,t_0]$ and $T^2\times[t_0-1,0]$ of opposite sign.  In any case, $T^2\times\{0\}$ is a mixed torus.\\

Once again we suppose that $(W,\omega)$ is an exact filling of $(M_n,\xi)$ and let $(M',\xi')$ be the exactly fillable contact manifold produced by Menke's decomposition theorem.  We write
\[
M' = S_0 \cup T^2\times I \cup S_1
\]
and let $\Gamma_{i},\mu_i$ denote the dividing curves and meridian of $\partial S_i$, for $i=0,1$.  Because $s_1=-1$, $\Gamma_0$ is represented by $(1,-1)^T\in T^2=\mathbb{R}^2/\mathbb{Z}^2$.  The shortest integer vectors representing $\Gamma_1$ and $\mu_1$ form an integral basis for $\mathbb{Z}^2$, so $\mu_1=(m,k)^T$, where $k=\pm 1-m$ and $m\geq 0$.  It follows that $\mu_0$ is represented by
\[
A\left(\begin{matrix} m\\ \pm 1-m\end{matrix}\right)
=-\left(\begin{matrix}
\pm n-m(n-1)\\ \pm 1-m
\end{matrix}\right).
\]
Our first claim is that $m\geq 1$.  If $m=0$, then the meridians $\mu_1$ and $\mu_0$ are represented by $(0,\pm 1)^T$ and $(\pm n,\pm 1)$, respectively.  Thus the counterclockwise rotation from $\mu_1$ to $\Gamma_1$ to $\Gamma_0$ to $\mu_0$ takes us from a slope of $\infty$ to a slope of $-1$ to a slope of $\frac{-1}{1-n}$, and then on to a slope of $1/n$.  In particular, the dividing curves of tori in $(M',\xi')$ rotate through an angle greater than $\pi$, and $(M',\xi')$ is overtwisted.  The case $m=0$ is depicted in Figure \ref{fig:2c-m-0}.\\

Next we claim that $k=1-m$.  Notice that for $m\geq 1$ and $n\leq -1$,
\[
\dfrac{-1-m}{m} < -1 < \dfrac{1+m}{n+m(n-1)} < \dfrac{-1}{1-n} < 0.
\]
So if $k=-1-m$, then
\[
\slope(\mu_{1}) < \slope(\Gamma_{1}) < \slope(\mu_{0}) < \slope(\Gamma_{0}) < 0.
\]
As before, this tells us that the contact planes of $(M',\xi')$ will rotate through an angle in excess of $\pi$, making $(M',\xi')$ overtwisted.  The case $k=-1-m$ is depicted in Figure \ref{fig:2c-k-minus}.\\

Finally, we claim that $m=1$.  If $m\geq 2$ and $n\leq -1$, then
\[
-1 < \dfrac{1-m}{m} \leq \dfrac{-1}{1-n} < \dfrac{1-m}{n-m(n-1)} < 0,
\]
which is to say
\[
\slope(\Gamma_{1}) < \slope(\mu_{1}) \leq \slope(\Gamma_{0}) < \slope(\mu_{0}) < 0,
\]
since $k=1-m$.  Once again, this causes $(M',\xi')$ to be overtwisted.  See Figure \ref{fig:2c-geq-2}.\\
	
At last we see that $m=1$ and $k=0$, so that $\mu_{1}$ is represented by $(1,0)^T$.  Then $\mu_{0}$ is represented by $A(1,0)^T=(-1,0)^T$, and $M'\simeq S^2\times S^1$.  But $S^2\times S^1$ admits a unique tight contact structure, and this structure has a unique strong symplectic filling up to symplectic deformation and blowup.  That is, the output $(W',\omega')$ of Menke's decomposition process is independent of the filling $(W,\omega)$ of $(M_n,\xi)$ with which we start.  Because the decomposition recovers $(W,\omega)$ from $(W',\omega')$, the filling $(W,\omega)$ is unique.
\end{proof}

\subsection{Hyperbolic torus bundles}\label{sec:hyperbolic}
Hyperbolic torus bundles represent the generic case for torus bundles, where the monodromy $A$ has $|\tr(A)|>2$.  The monodromy has the form
\[
A = \pm T^{r_0}ST^{r_1}S\cdots T^{r_k}S,
\]
where $r_0\leq -3$ and $r_i\leq -2$ for $1\leq i\leq k$ and $S, T$ are the generators of $SL(2,\mathbb{Z})$ identified in Section \ref{sec:background}.  Honda showed in \cite{honda2000classification2} that the torus bundle with this monodromy admits $|(r_0+1)\cdots(r_k+1)|$ minimally twisting tight contact structures.  If the monodromy is positive, then two of these structures are universally tight; otherwise they are all virtually overtwisted.\\

We will find it convenient to change our monodromy by a conjugation, and to relabel our coefficients.  Given $a\leq -3$ and $\tau\geq 0$, set
\[
C_{a,\tau} = T^{a+1}S(T^{-2}S)^\tau T^{-1}.
\]
We determine the monodromy $A$ by choosing integers $a_0,\ldots,a_k\leq -3$ and $\tau_0,\ldots,\tau_k\geq 0$ and setting
\[
A = \pm C_{a_0,\tau_0}\cdots C_{a_k,\tau_k} = \pm T^{a_0+1}S(T^{-2}S)^{\tau_0} \cdots T^{a_k}S(T^{-2}S)^{\tau_k} T^{-1}.
\]
Then $T^{-1}AT$ has the form identified above.  Notice that with this notation the count of tight contact structures is $|(a_0+1)\cdots(a_k+1)|$.  The primary benefit to this new notation is that we may easily identify a basic slice decomposition of $(M,\xi)$.  Namely, our decomposition will have $|a_0+\cdots+a_k+2(k+1)|$ basic slices, divided into $k+1$ continued fraction blocks.\\

With this notation established, we may state more explicitly our result for hyperbolic torus bundles.

\begin{proposition}\label{hyperbolic-result}
Let $M$ be a hyperbolic torus bundle and let $\xi$ be a virtually overtwisted tight contact structure on $M$.  Then there is a nonempty, finite list $(L_1,\xi_1),\ldots,(L_m,\xi_m)$ of tight lens spaces and a corresponding list $L^1_\pm,\ldots,L^m_\pm$ of Legendrian knots $L^i_\pm\subset(L_i,\xi_i)$ such that every strong (exact) symplectic filling of $(M,\xi)$ can be obtained from a strong (exact) symplectic filling of $(L_i,\xi_i)$, for some $1\leq i\leq m$, by attaching a round symplectic 1-handle along $L^i_\pm$.  Moreover,
\begin{enumerate}
	\item if $(M,\xi)$ has a virtually overtwisted continued fraction block, we have $m=1$;\label{vot-cfb}
	\item if the monodromy of $M$ has coefficients $a_0,a_1,\ldots,a_k$ and $\tau_0,\ldots,\tau_k$, then we have $m\leq 2+\max\{\tau_i\}$.\label{hyperbolic-ub}
\end{enumerate}
\end{proposition}

\begin{remark}
The lists $(L_1,\xi_1),\ldots,(L_m,\xi_m)$ and $L^1_\pm,\ldots,L^m_\pm$ are determined by the choice of a mixed torus in $(M,\xi)$.  We will construct this list for each mixed torus in $(M,\xi)$, and prove the last two statements of the proposition by showing that $(M,\xi)$ admits a mixed torus leading to a list of the desired length.
\end{remark}

As in the previous cases, our strategy is to identify a mixed torus in $(M,\xi)$, and then to determine which lens spaces may result from cutting $(M,\xi)$ open along this torus and gluing on solid tori.  We then use Theorem \ref{thm:jsj} to conclude that all fillings of $(M,\xi)$ result from fillings of these lens spaces via round symplectic 1-handle attachment.  Notice that Theorem \ref{thm:jsj} specifies the attaching regions for these lens spaces.  We will follow Honda \cite{honda2000classification2} in identifying $M$ with a quotient of $T^2\times I$, where $T^2\times\{1\}$ has slope $s_1=\infty$.  The basic slice decomposition of $T^2\times I$ will then have $k+1$ continued fraction blocks, with a block corresponding to each coefficient $a_i$.  The continued fraction block associated to $a_i$ will consist of $|a_i+2|$ basic slices, and will have $|a_i+1|$ tight contact structures, distinguished by the number of basic slices in the block which have positive Euler characteristic.\\

Because there are $|a_i+1|$ tight contact structures on the continued fraction block associated with $a_i$, we see that there are $|(a_0+1)\cdots(a_k+1)|$ tight contact structures on $T^2\times I$ with the desired boundary slopes.  Honda shows in \cite{honda2000classification2} that these pass to distinct tight contact structures on $M$.  Exactly two of the structures on $T^2\times I$ are universally tight --- those two for which every basic slice has the same sign.  If our monodromy is positive, these two structures remain universally tight when we pass to $M$; if $A$ is negative, these structures become virtually overtwisted.\\

Our basic slice decomposition of $M$ shows us that there are $\ell=|a_0+\cdots+a_k+2(k+1)|$ tori which appear along the boundary of a basic slice.  If $(M,\xi)$ is virtually overtwisted, then at least one of these $\ell$ tori sits between basic slices of opposite sign, and is thus a mixed torus.  We choose such a torus $T^2\times\{t_0\}$ and call it $T^2$.\\

Now $A$ has an oriented eigenbasis $\{v_1,v_2\}$ with associated eigenvalues $\lambda_1>1$ and $0<\lambda_2<1$.  These vectors are necessarily irrational, and we denote their slopes by
\[
\Lambda^s := \slope(v_1)
\quad\text{and}\quad
\Lambda^u := \slope(v_2).
\]
These are the \emph{stable} and \emph{unstable slopes}, respectively.  On the Farey tessellation, the slopes corresponding to the dividing sets of fibers of our $T^2\times I$ are located in the counterclockwise arc connecting $\Lambda^u$ to $\Lambda^s$.  In particular, $\slope(T^2)$ and $\slope(AT^2)$ are in this sector.\\

The slopes $\slope(T^2)$ and $\slope(AT^2)$ will play the roles played by $s_1$ and $s_0$, respectively, in previous cases.  Namely, the lens space $(M',\xi')$ that results from Menke's JSJ decomposition will have the form
\[
S_{t_0} \cup (T^2\times[t_0,t_0+1]) \cup S_{t_0+1},
\]
with boundary slopes $s_{t_0+1}=\slope(T^2)$ and $s_{t_0}=\slope(AT^2)$.  We now identify the possible slopes for the meridian $\mu$ of $S_{t_0+1}$.  Certainly $s(\mu)$ must be connected to $s_{t_0+1}$ by an edge on the Farey tessellation, since the shortest integral vectors representing these slopes form an integral basis for $\mathbb{Z}^2$.  Our next claim is that $s(\mu)$ must not lie in the same sector of the Farey tessellation as $T^2\times I$.

\begin{lemma}\label{lemma:hyperbolic-meridian-arc}
On the Farey tessellation, $s(\mu)$ lies in the counterclockwise arc connecting $\Lambda^s$ to $\Lambda^u$.
\end{lemma}
\begin{proof}
As in the previous cases, rotating from $s(\mu)$ to $s_{t_0+1}$ to $s_{t_0}$ to $s(A\mu)$ must not take us through an angle in excess of $\pi$, lest $(M',\xi')$ be overtwisted.  On the Farey tessellation, this means that the counterclockwise path connecting these slopes (in this order) must not overlap itself.  Now if $s(\mu)$ lies on the counterclockwise arc $[s(T^2),\Lambda^s]$ between $s(T^2)$ and $\Lambda^s$, then $s(A\mu)$ lies on $[s(\mu),\Lambda^s]$.  But then the arc $[s(\mu),s(T^2)]$ contains $s(A\mu)$, meaning that the rotation described above is through an angle greater than $\pi$, and $(M',\xi')$ is overtwisted.  On the other hand, if $s(\mu)$ lies on $[\Lambda^u,s(T^2)]$, then $s(A\mu)$ is contained in $[s(\mu),s(AT^2)]$.  Again this means that our path of slopes overlaps itself, and $(M',\xi')$ is overtwisted.  We conclude that $s(\mu)$ lies on $[\Lambda^s,\Lambda^u]$.
\end{proof}

With Lemma \ref{lemma:hyperbolic-meridian-arc} in hand, we quickly obtain the finite list $(L_1,\xi_1),\ldots,(L_m,\xi_m)$ guaranteed by Proposition \ref{hyperbolic-result}.  Indeed, suppose we have a sequence $(s(\mu_i))$ of meridian slopes connected to $s(T^2)$ on the Farey tessellation (as all candidate meridian slopes must be).  Then this sequence converges to $s(T^2)$, and thus only finitely many of the slopes are contained in $[\Lambda^s,\Lambda^u]$, since $s(T^2)$ is not contained in this interval.  So there are only finitely many possible meridians $\mu$ for $S_{t_0+1}$, and hence only finitely many lens spaces to which $M'$ could be diffeomorphic.  Notice that the contact structure $\xi'$ on $M'$ is determined by $\xi$, so we now know that we have a finite list of tight lens spaces.  Moreover, each of these lens spaces has a pair $L^i_\pm\subset(L_i,\xi_i)$ of distinguished Legendrian knots, along which we attach symplectic 1-handles to produce fillings of $(M,\xi)$.  Each of these is constructed along with its lens space as the core curves of $S_{t_0}$ and $S_{t_0+1}$, respectively.  We will obtain parts (\ref{vot-cfb}) and (\ref{hyperbolic-ub}) of Proposition \ref{hyperbolic-result} by examining the slopes which appear in a continued fraction block.

\begin{figure}
\centering
\input{cfb.tex}
\caption{The interval $[\Lambda^s,\Lambda^u]$ for $A=C_{-5,2}$.  Notice that $s_1=\infty$ is connected to $0,1,2,3$, while 0 is the only element of $[\Lambda^s,\Lambda^u]$ to which either of $-1$ or $-1/2$ is connected.}
\label{fig:cfb}
\end{figure}
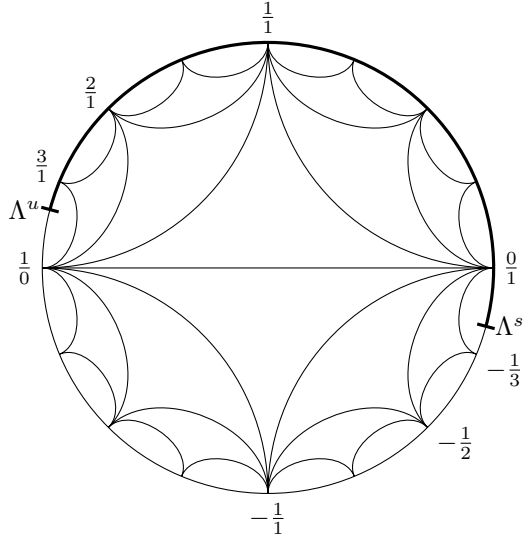

\begin{lemma}\label{lemma:cfb}
Choose $a\leq -3$ and $\tau\geq 0$, and let $A=C_{a,\tau}$, with stable and unstable slopes $\Lambda^s$ and $\Lambda^u$.  Let $(T^2\times I,\xi)$ be tight and have boundary slopes $s_1$ and $s_0$ corresponding to $(0,1)^T$ and $A(0,1)^T$, respectively, and let $T^2\subset T^2\times I$ be a boundary component for a basic slice.  Then
\begin{enumerate}
	\item if $T^2$ is not a boundary torus, there is exactly one slope $s$ connected to $s(T^2)$ in the arc $[\Lambda^s,\Lambda^u]$ on the Farey tessellation;\label{lemma:unique-slope}
	\item if $T^2=T^2\times\{1\}$, there are $\tau+2$ slopes connected to $s(T^2)$ in the arc $[\Lambda^s,\Lambda^u]$.\label{lemma:boundary-torus}
\end{enumerate}
\end{lemma}
\begin{proof}
First, note that the boundary torus $T^2\times\{1\}$ has slope $\infty$, and that $(T^2\times I,\xi)$ contains $|a+2|$ basic slices.  In particular, there are $|a+3|$ non-boundary tori which lie between basic slices, and these have slopes
\[
-1,-1/2,\ldots,1/(a+3).
\]
We will prove this lemma by showing that $0$ is the only value in $[\Lambda^s,\Lambda^u]$ that is connected to any of these slopes on the Farey tessellation, and by counting the values in $[\Lambda^s,\Lambda^u]$ which are connected to $\infty$.  Note that expanding $A$ gives
\[
A = \left(\begin{matrix}
-\tau-(\tau+1)(1+a) & a+2\\
-(\tau+1) & 1
\end{matrix}\right),
\]
and thus $T^2\times\{0\}$ has slope $\frac{1}{a+2}$.  In particular, since $[\Lambda^s,\Lambda^u]\subset[s_0,s_1]$, we have $\frac{1}{a+2}<\Lambda^s$.  Similarly, because
\[
A^{-1} = \left(\begin{matrix}
1 & -a-2\\
\tau+1 & -\tau-(\tau+1)(1+a)
\end{matrix}\right),
\]
$A^{-1}(0,1)^T$ has slope $1+\tau+\frac{1}{|a+2|}>\Lambda^u$.  At the same time,
\[
A\left(\begin{matrix} 1\\ \Lambda^u\end{matrix}\right) = \lambda\left(\begin{matrix} 1\\ \Lambda^u\end{matrix}\right)
\]
for some $0<\lambda<1$.  An explicit calculation shows that
\[
\Lambda^u = \frac{1}{2}(1+\tau+\sqrt{(1+\tau)^2+4(1+\tau)/|a+2|}) > 1+\tau.
\]
So $1+\tau < \Lambda^u < 1+\tau+\frac{1}{|a+2|} < 2+\tau$.  Similar reasoning shows that $\frac{1}{a+2}<\Lambda^s<0$.  On the Farey tessellation, $\infty$ is connected only to integers, and we see that $[\Lambda^s,\Lambda^u]$ contains precisely $\tau+2$ integers: $0,1,\ldots,\tau,1+\tau$.  So $s_1$ is connected to $\tau+2$ slopes in $[\Lambda^s,\Lambda^u]$, proving part (\ref{lemma:boundary-torus}).  We also notice that each of the slopes $-1,-1/2,\ldots,1/(a+3)$ is connected to 0 on the Farey tessellation, and, with the exception of $-1$, is connected only to rational numbers $p/q<1/(a+2)$.  In particular, 0 is the only value in $[\Lambda^s,\Lambda^u]$ to which any of the slopes $-1,-1/2,\ldots,1/(a+3)$ is connected.  See Figure \ref{fig:cfb}.  This proves part (\ref{lemma:unique-slope}).
\end{proof}

While the precise slopes identified by Lemma \ref{lemma:cfb} change when we compose to obtain
\[
A=\pm C_{a_0,\tau_0}\cdots C_{a_k,\tau_k},
\]
the counts do not.  Indeed, each $C_{a_i,\tau_i}$ is an element of $SL(2,\mathbb{Z})$ and will therefore preserve connections and order on the Farey tessellation.  So while the finite list of slopes $s_1=\infty,-1,-1/2,\ldots,1/(a+3),1/(a+2)=s_0$ and the interval $[\Lambda^s,\Lambda^u]$ will change after composition, the connections between the former and elements of the latter will be unaltered.\\

In particular, if $(M,\xi)$ has a virtually overtwisted continued fraction block, then there is a mixed torus $T^2$ which is interior to this block.  According to Lemma \ref{lemma:cfb}, applying the JSJ decomposition along this mixed torus leaves us with precisely one possible meridian $\mu$ for $S_{t_0+1}$.  Another possibility is for $(M,\xi)$ to be virtually overtwisted, but to have continued fraction blocks which are all universally tight.  In this case we find a mixed torus $T^2$ which lies between two continued fraction blocks.  Lemma \ref{lemma:cfb} tells us that $s(T^2)$ is connected to at most $2+\max\{\tau_i\}$ distinct slopes in $[\Lambda^s,\Lambda^u]$, and thus there are at most $2+\max\{\tau_i\}$ distinct lens spaces which may result from applying the JSJ decomposition along $T^2$.  This completes the proof of Proposition \ref{hyperbolic-result}.

\begin{example}
Let $M$ be the positive hyperbolic torus bundle with coefficients $(a_0,a_1,a_2)=(-4,-5,-4)$ and $(\tau_0,\tau_1,\tau_2)=(0,2,0)$.  That is, $M$ has monodromy
\[
A = T^{-4}ST^{-5}ST^{-2}ST^{-2}ST^{-4}S = \left(\begin{matrix}
119 & -83\\ -43 & 30
\end{matrix}\right).
\]
According to the classification of tight contact structures on torus bundles (\cite{honda2000classification2}), the tight contact structures on $M$ are in one-to-one correspondence with the tight contact structures on $T^2\times I$ which have boundary slopes $s_0=-1$ and $s_1=[-4,-2,-2,-5,-3]=-119/36$.  Each such tight structure decomposes into seven basic slices, distributed among three continued fraction blocks, visualized as follows:
\begin{center}
\input{interval.tex}
\end{center}
The long tick marks at slopes $-1,-3,-33/10$, and $-119/36$ indicate divisions between continued fraction blocks.  The tori $T^2\times\{0\}$ and $T^2\times\{1\}$ (whose dividing sets have slopes $-1$ and $-119/36$, respectively) will be identified by the monodromy $A$.\\

We now determine a tight structure on $M$ by decorating each basic slice with a sign, indicating the sign of the relative Euler class of the tight structure when restricted to this basic slice.  Because basic slices may be shuffled within a given continued fraction block, there are $3\cdot 4\cdot 3=36$ tight contact structures on $M$.  As observed by Bhupal-Ozbagci in \cite{bhupal2014canonical}, each of these structures admits a Stein filling, depicted in Figure \ref{fig:example-filling}.\\

\begin{figure}
\centering
\def\svgwidth{0.8\textwidth}
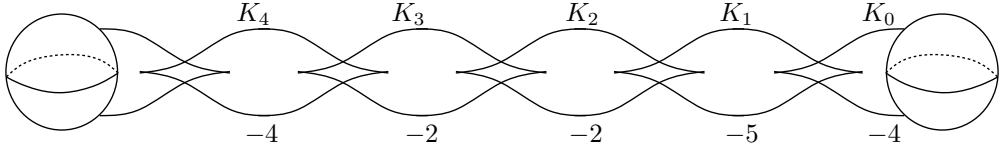
\caption{The standard filling of any tight structure on the torus bundle with positive monodromy coefficients $(r_0,r_1,r_2,r_3,r_4)=(-4,-5,-2,-2,-4)$ is obtained from Legendrian surgery along the above link in $(S^1\times S^2,\xi_{std})$.  Each knot $K_i$ is stabilized so that its Thurston-Bennequin number exceeds its framing by 1, and the rotation numbers are determined by the tight structure $\xi$ on $M$.  The lens spaces produced by Theorem \ref{main-theorem} may be obtained by erasing the 1-handle and an unknot, leaving a link in $(S^3,\xi_{std})$.}
\label{fig:example-filling}
\end{figure}

Each of the three continued fraction blocks in $M$ has a corresponding tight lens space.  Stein fillings of these lens spaces may be obtained from Figure \ref{fig:example-filling} by erasing the 1-handle along with one of the unknots $K_0$, $K_1$, or $K_4$.  So the lens spaces associated to the continued fraction blocks have topological types $L(43,13)$, $L(37,10)$, and $L(49,34)$, respectively.  The tight structure on the resulting lens space is determined by Figure \ref{fig:example-filling}.  If a given continued fraction block in $(M,\xi)$ is virtually overtwisted, then $(M,\xi)$ admits a mixed torus interior to this continued fraction block, and each exact filling of $(M,\xi)$ may be obtained from an exact filling of the corresponding lens space by round 1-handle attachment.\\

There are eight tight contact structures on $M$ which are virtually overtwisted, but for which all three continued fraction blocks are universally tight.  In these cases, $(M,\xi)$ admits a mixed torus which lies at the boundary of two continued fraction blocks, and the list of lens spaces produced by Theorem \ref{main-theorem} consists of the lens spaces corresponding to these two continued fraction blocks, as well as all lens spaces \emph{between} these.  These extra lens spaces are obtained from Figure \ref{fig:example-filling} by deleting the round 1-handle and an unknot with a $-2$-framing.  For instance, if we have a mixed torus of slope $-33/10$ --- sitting at the boundary of the second and third continued fraction blocks --- then the list of lens spaces produced by Theorem \ref{main-theorem} has topological type
\[
(L(37,10),L(123,26),L(127,71),L(49,34)).
\]
A mixed torus between the first and second or third and first continued fraction blocks will yield a list of exactly two lens spaces, as there are no intermediate lens spaces in these cases.\\

The following diagram summarizes this example:
\begin{center}
\input{example-diagram.tex}
\end{center}
If $\xi$ is a virtually overtwisted tight contact structure on $M$, then there must be a mixed torus at one of the seven slopes above.  (Because $T^2\times\{1\}$ is identified with $T^2\times\{0\}$, we need not consider this torus separately.)  Below each slope we see the diffeomorphism types of the lens spaces produced by Theorem \ref{main-theorem} when a mixed torus occurs with this slope.  As mentioned above, the tight structures on these lens spaces will be determined by $\xi$, as depicted in Figure \ref{fig:example-filling}.
\end{example}

\section{Distinct decompositions of fillings}\label{sec:overcounting}
Without a full classification of symplectic fillings of lens spaces, Theorem \ref{main-theorem} cannot be used to completely classify the fillings of virtually overtwisted hyperbolic torus bundles.  Even with such a classification for lens spaces, the only immediate consequence of Theorem \ref{main-theorem} is a recipe which is guaranteed to construct all fillings of such a torus bundle, but these constructions need not be unique.  Indeed, if $(L(p_1,q_1),\xi_1),\ldots,(L(p_m,q_m),\xi_m)$ is the list of lens spaces provided by Theorem \ref{main-theorem} for some hyperbolic torus bundle $(M,\xi)$, it is possible that there are fillings $(W_i,\omega_i)$ of $(L(p_i,q_i),\xi_i)$ for $i=1,\ldots,m$, each of which yield the same filling $(W,\omega)$ of $(M,\xi)$ after round 1-handle attachment --- in fact, we saw in Section \ref{sec:intro} that this is the case for the standard fillings of the lens spaces.\\

To fully explicate such overcounting of the fillings of $(M,\xi)$ will require a detailed understanding of the fillings of $(L(p_1,q_1),\xi_1),\ldots,(L(p_m,q_m),\xi_m)$.  The proof of Theorem \ref{main-theorem} establishes a relationship among the lens spaces $(L(p_1,q_1),\xi_1),\ldots,(L(p_m,q_m),\xi_m)$ that allows us to identify one source of overcounting.  In particular, the construction associates to each lens space $(L(p_i,q_i),\xi_i)$ a pair of distinguished Legendrian knots $L^i_-,L^i_+\subset(L(p_i,q_i),\xi_i)$ which arise as the core curves of the solid tori which are glued onto $T^2\times\{0\}$ and $T^2\times\{1\}$, respectively, after $(M,\xi)$ is cut open along its mixed torus.  For $1\leq i\leq m-1$, we obtain the lens space $(L(p_{i+1},q_{i+1}),\xi_{i+1})$ by simultaneously performing $(+1)$-surgery along $L^i_+$ and $(-1)$-surgery along $L^i_-$.\\

This relationship between the lens spaces gives us an algorithm for building fillings of $(L(p_{i+1},q_{i+1}),\xi_{i+1})$ from fillings of $(L(p_i,q_i),\xi_i)$.  If $(W_i,\omega_i)$ is a strong symplectic filling of $(L(p_i,q_i),\xi_i)$ in which $L^i_+$ bounds a Lagrangian disk, then we may remove a neighborhood of the disk from $(W_i,\omega_i)$ to obtain a symplectic manifold strongly filling its boundary (c.f. \cite[Theorem 3.1]{conway2017symplectic}), and the effect on the boundary is to perform $(+1)$-surgery along $L^i_+$.  We may then attach a Weinstein 2-handle to this new filling along $L^i_{-}$, and at the boundary this has the effect of performing a $(-1)$-surgery along $L^i_{-}$.  The result is $(W_{i+1},\omega_{i+1})$, a strong symplectic filling of $(L(p_{i+1},q_{i+1}),\xi_{i+1})$, and Proposition \ref{prop:move-2-handle} will tell us that this filling leads to the same filling of $(M,\xi)$ as does $(W_i,\omega_i)$ after round 1-handle attachment.\\

To make the statement of Proposition \ref{prop:move-2-handle} less cumbersome, we establish some notation.  Suppose we have a strong symplectic filling $(X,\omega)$ of a contact manifold $(M,\xi)$, with Legendrian knots $L_0,L_1\subset(M,\xi)$.  Moreover, suppose that $L_0$ is the boundary of a Lagrangian disk $D\subset(X,\omega)$ which meets $\partial X$ transversely.  As alluded to above, we may identify a neighborhood $N\subset X$ of $D$ with a neighborhood of the cocore of a Weinstein 2-handle, and remove this neighborhood to obtain a new symplectic filling $(\overline{X},\overline{\omega})$.  (The existence of such a neighborhood is proven in Theorem 3.1 of \cite{conway2017symplectic}.)  With $(X,\omega)$, $D$, $L_0$, and $L_1$ understood, we denote by $(X',\omega')$ the strong symplectic filling which results from attaching a Weinstein 2-handle to $(\overline{X},\overline{\omega})$ along $L_1$.  Each of $L_0$ and $L_1$ has a natural Legendrian pushoff in $\partial(X',\omega')$, given by the boundary of the core and cocore  of the associated 2-handle, respectively, which we denote $L'_0$ and $L'_1$.  With this notation in hand, we have the following result.

\begin{proposition}\label{prop:move-2-handle}
Let $(X,\omega)$ and $(X',\omega')$ be strong symplectic fillings related by the construction above.  Let $(W,\omega_W)$ be the strong symplectic filling obtained by attaching a round symplectic 1-handle to $(X,\omega)$ along $L_0,L_1$, and let $(W',\omega'_{W})$ be the analogous filling for $(X',\omega')$.  Then $(W,\omega)$ and $(W',\omega'_{W})$ are symplectomorphic fillings of $(M,\xi)$.
\end{proposition}

Our strategy for proving this result is to realize round symplectic 1-handle attachment as a sequence of Weinstein handle attachments, and then to reorder the Weinstein handles.  The decomposition of a round 1-handle into Weinstein handles is explained in great generality in \cite[Section 7.2]{avdek2012liouville}, but the precise statement we need here is the following.

\begin{lemma}\label{lemma:round-is-weinstein}
Fix a strong symplectic filling $(X,\omega)$ of a contact 3-manifold $(M,\xi)$, and identify Legendrian knots $L_0,L_1\subset(M,\xi)$.  Consider the following symplectic fillings:
\begin{enumerate}
	\item $(X_r,\omega_r)$, obtained by attaching a round symplectic 1-handle to $(X,\omega)$ along the knots $L_0,L_1\subset (M,\xi)$;
	\item $(X_w,\omega_w)$, obtained by attaching a Weinstein 1-handle to $(X,\omega)$ along points $p_i\in L_i$, $i=0,1$, and then attaching a Weinstein 2-handle to the resulting filling along the knot $L$ obtained by surgering $L_0$ and $L_1$ along $p_0$ and $p_1$.
\end{enumerate}
Then $(X_r,\omega_r)$ is symplectomorphic to $(X_w,\omega_w)$.
\end{lemma}

Before providing a proof of Lemma \ref{lemma:round-is-weinstein}, let us say what we mean by \emph{surgering $L_0$ and $L_1$ along $p_0$ and $p_1$}.  A 4-dimensional Weinstein 1-handle $H_1$ admits a Lagrangian submanifold-with-boundary $\Lambda\subset H_1$ to which the Liouville vector field on $H_1$ is tangent, and such that this vector field gives $\Lambda$ the structure of a 2-dimensional Weinstein 1-handle.  We then have Legendrian submanifolds $\partial_{in}\Lambda\subset\partial_{in}H_1$ and $\partial_{out}\Lambda\subset\partial_{out}H_1$.  Now we attach $H_1$ to $(X,\omega)$ along $p_0,p_1$ by choosing a contactomorphism from the attaching region $\partial_{in}H_1$ to a neighborhood of the points $p_0,p_1$.  This contactomorphism can be chosen so that $\partial_{in}\Lambda\subset\partial_{in}H_1$ is mapped to arcs $a_0,a_1$ of the Legendrians $L_0,L_1$.  In the boundary of the symplectic filling that results, we will have a Legendrian knot
\[
L = (L_0\setminus a_0) \cup \partial_{out}\Lambda \cup (L_1\setminus a_1),
\]
and it is this knot that we consider to be the result of surgering $L_0$ and $L_1$ along $p_0$ and $p_1$.

\begin{proof}[Proof of Lemma \ref{lemma:round-is-weinstein}]
Once we identify our round symplectic 1-handle with a symplectic handle in the sense of Avdek, this fact follows from the discussion in \cite[Section 7.2]{avdek2012liouville}.  We will present Avdek's argument in our particular case.\\

As mentioned in Section \ref{sec:background-jsj}, Avdek defines an abstract round symplectic 1-handle as follows.  Let $(\Sigma,\beta)=(D^*S^1,\lambda_{can})$ be the unit disk bundle in $(T^*S^1,\lambda_{can})$, and consider the contact manifold
\begin{equation}\label{eq:nbhd-of-liouville}
(N(\Sigma)=[-\epsilon,\epsilon]\times\Sigma,\alpha=dz+\beta),
\end{equation}
where $z$ is the coordinate on $[-\epsilon,\epsilon]$.  Avdek rounds the edges of $N(\Sigma)$ to obtain $\mathcal{N}(\Sigma)$, and then defines the symplectic manifold
\begin{equation}\label{eq:symplectic-handle}
(H_\Sigma,\omega_\beta) = ([-1,1]\times\mathcal{N}(\Sigma),d\theta\wedge dz+d\beta),
\end{equation}
where $\theta$ is the coordinate on $[-1,1]$.  After more edge-rounding, this is an abstract copy of a round symplectic 1-handle, to be attached along the ends $\{\pm 1\}\times\mathcal{N}(\Sigma)$.  In particular, we have the Liouville form
\[
\lambda_\Sigma = -\theta dz - 2zd\theta + \beta,
\]
and $(H_\Sigma,\lambda_\Sigma)$ carries a Liouville vector field $Z$ which points into $H_\Sigma$ along $\{\pm 1\}\times\mathcal{N}(\Sigma)$ and out of $H_\Sigma$ along $[-1,1]\times\partial\mathcal{N}(\Sigma)$.\\

Now $(\Sigma,\beta)$ has an obvious handle decomposition as a Weinstein domain, given by attaching a Weinstein 1-handle to a Weinstein 0-handle, and yielding the filtration
\[
(\mathbb{D}^2,\lambda_{std})=(\Sigma_0,\beta_0)\subset(\Sigma,\beta)
\]
of $(\Sigma,\beta)$.  By carrying out the constructions of (\ref{eq:nbhd-of-liouville}) and (\ref{eq:symplectic-handle}) for $(\Sigma_0,\beta_0)$, we obtain filtrations
\[
(\mathbb{D}^3,\alpha_{std})=(\mathcal{N}(\Sigma_0),\alpha_0) \subset (\mathcal{N}(\Sigma),\alpha)
\quad\text{and}\quad
(H_{\Sigma_0},\lambda_{\Sigma_0}) \subset (H_\Sigma,\lambda_\Sigma),
\]
and it is clear that $(H_{\Sigma_0},\lambda_{\Sigma_0})$ is symplectomorphic to a Weinstein 1-handle.  It remains to verify that we obtain $(H_\Sigma,\lambda_\Sigma)$ from $(H_{\Sigma_0},\lambda_{\Sigma_0})$ by attaching a Weinstein 2-handle.\\

To this end, we identify a Legendrian ribbon $(\Sigma^i,\beta^i)$ of $L_i\subset(M,\xi)$ with $(\Sigma,\beta)$, using the filtration of $(\Sigma,\beta)$ to define $(\Sigma^i_0,\beta^i_0)\subset(\Sigma^i,\beta^i)$, for $i=0,1$.  Now the Liouville hypersurfaces $(\Sigma^0,\beta^0)$ and $(\Sigma^1,\beta^1)$ admit standard neighborhoods $N(\Sigma^0),N(\Sigma^1)\subset(M,\xi)$ along which a round symplectic 1-handle may be attached to $(X,\omega)$ to yield $(X_r,\omega_r)$.  On the other hand, let $(X_1,\omega_1)$ denote the result of attaching the Weinstein 1-handle $(H_{\Sigma_0},\lambda_{\Sigma_0})$ to $(X,\omega)$ along $N(\Sigma^0_0)$ and $N(\Sigma^1_0)$.  By definition, we obtain $(X_w,\omega_w)$ from $(X_1,\omega_1)$ by attaching a Weinstein 2-handle, but the filtration $(H_{\Sigma_0},\lambda_{\Sigma_0})\subset(H_\Sigma,\lambda_\Sigma)$ above allows us to view $(X_1,\omega_1)$ as living inside of $(X_r,\omega_r)$.  We will use this perspective to see $(X_r,\omega_r)$ as the result of attaching a Weinstein 2-handle to $(X_1,\omega_1)$ and thus obtain the desired symplectomorphism.\\

For $i=0,1$, let $\Lambda_i\subset L_i\subset(\Sigma^i,\beta^i)$ be the core disk of the 1-handle attached to $(\Sigma^i_0,\beta^i_0)$ to yield $(\Sigma^i,\beta^i)$.  This is a Legendrian chord in the boundary of $(X_1,\omega_1)$, and we identify $\Lambda_i$ inside of $N(\Sigma^i)$ as
\[
\Lambda_i = \{z=0\}\times\Lambda_i \subset [-\epsilon,\epsilon]\times\Sigma^i = N(\Sigma^i).
\]
At the same time, consider the disk
\[
\tilde{\Lambda} = [-1,1]\times\Lambda \subset (H_{\Sigma}\setminus\Int H_{\Sigma_0}),
\]
where $\Lambda\subset(\Sigma,\beta)$ is the analogous chord in $\Sigma$.  Viewing $(X_1,\omega_1)$ as a subset of $(X_r,\omega_r)$, we see that $(X_1,\omega_1)\cap\tilde{\Lambda}=\Lambda_0\sqcup\Lambda_1$, and that $\tilde{\Lambda}$ represents the core disk of the Weinstein 2-handle $(H_\Sigma \setminus\Int H_{\Sigma_0},\lambda_{\Sigma})$.  The latter statement follows from the fact that $\Lambda$ is the core disk of the Weinstein 1-handle $(\Sigma\setminus\Sigma_0,\beta)$, meaning that $\beta$ vanishes along $\Lambda$, and thus $\lambda_\Sigma$ vanishes along $\tilde{\Lambda}$.  Up to smoothing, this disk has boundary
\[
\partial\tilde{\Lambda} = \Lambda_0 \cup ([-1,1]\times\partial\Lambda) \cup \Lambda_1 \subset (X_1,\omega_1),
\]
which is the result in $(X_1,\omega_1)$ of surgering $L_0$ to $L_1$.  So $(X_r,\omega_r)$ is obtained from $(X_1,\omega_1)$ by attaching a Weinstein 2-handle along this surgered knot, as desired, and thus round 1-handle attachment may be realized as Weinstein 1-handle attachment followed by Weinstein 2-handle attachment.
\end{proof}

We are now prepared to prove Proposition \ref{prop:move-2-handle}.

\begin{proof}[Proof of Proposition \ref{prop:move-2-handle}]
We continue using the notation established before the statement of Proposition \ref{prop:move-2-handle}.  In particular, we have a symplectic filling $(\overline{X},\overline{\omega})$ with Legendrian knots $\overline{L}_0,\overline{L}_1\subset\partial(\overline{X},\overline{\omega})$ such that $(X,\omega)$ and $(X',\omega')$ are obtained from $(\overline{X},\overline{\omega})$ by attaching a Weinstein 2-handle along $\overline{L}_0$ and $\overline{L}_1$, respectively.\\

Now since Lemma \ref{lemma:round-is-weinstein} tells us that round symplectic 1-handle attachment consists of a Weinstein 1-handle attachment followed by a Weinstein 2-handle attachment, we see that $(W,\omega)$ and $(W',\omega')$ are each obtained from $(\overline{X},\overline{\omega})$ by a sequence of Weinstein handle attachments.  In particular, we obtain $(W,\omega)$ by attaching a Weinstein 2-handle along $\overline{L}_0$, then attaching a Weinstein 1-handle along points in $L_0,L_1\subset\partial(X,\omega)$ --- knots which are Legendrian pushoffs of $\overline{L}_0,\overline{L}_1$ --- and then attaching a Weinstein 2-handle along the resulting surgered knot.  By considering $\overline{L}_1$ instead of $\overline{L}_0$, we obtain $(W',\omega')$.  In either case, we may reorder our handle attachments so that the round symplectic 1-handle is attached first to produce a filling $(\overline{X}',\overline{\omega}')$, leaving us to attach a Weinstein 2-handle along either $\overline{L}'_0$ or $\overline{L}'_1$ --- these being Legendrian pushoffs of the corresponding knots in $(\overline{X},\overline{\omega})$.  But since we have attached a round symplectic 1-handle along $\overline{L}_0,\overline{L}_1$, their pushoffs are Legendrian isotopic in $\partial(\overline{X}',\overline{\omega}')$, with the isotopy being witnessed by a cylinder in the boundary of the round 1-handle.  So $(W,\omega)$ and $(W',\omega')$ are obtained from $(\overline{X},\overline{\omega})$ by the same sequence of Weinstein handle attachments, and thus are symplectomorphic.
\end{proof}

The practical upshot of Proposition \ref{prop:move-2-handle} is this: for $1\leq i\leq m-1$, the lens space $(L(p_i,q_i),\xi_i)$ produced by Theorem \ref{main-theorem} comes with distinguished Legendrian knots $L^i_-,L^i_{+}\subset L(p_i,q_i)$ which are used to produce $(L(p_{i+1},q_{i+1}),\xi_{i+1})$ as described above.  Proposition \ref{prop:move-2-handle} says that if we have a filling of $(L(p_i,q_i),\xi_i)$ in which $L^i_+$ bounds a Lagrangian disk, then the filling of $(M,\xi)$ produced by round 1-handle attachment is also obtained from the corresponding filling of $(L(p_{i+1},q_{i+1}),\xi_{i+1})$.  In particular, if we are using Theorem \ref{main-theorem} to tabulate the fillings of $(M,\xi)$, then any filling of $(L(p_i,q_i),\xi_i)$ in which $L^i_+$ bounds a Lagrangian disk may be ignored, as there is a filling of $(L(p_{i+1},q_{i+1}),\xi_{i+1})$ which will produce the same filling of $(M,\xi)$.  Concretely, we have the following corollary.

\begin{corollary}\label{cor:overcounting}
Let $(M,\xi)$ be a virtually overtwisted hyperbolic torus bundle, and let
\[
(L(p_1,q_1),\xi_1),\ldots,(L(p_m,q_m),\xi_m)
\]
be the list of tight lens spaces produced by Theorem \ref{main-theorem}, with distinguished Legendrian knots $L^1_-,L^1_+,\ldots,L^m_-,L^m_+$.  For $1\leq i\leq m-1$, if $(W_i,\omega_i)$ is a strong symplectic filling of $(L(p_i,q_i),\xi_i)$ in which $L^i_+$ bounds a Lagrangian disk, then there is a strong filling $(W_{i+1},\omega_{i+1})$ of $(L(p_{i+1},q_{i+1}),\xi_{i+1})$ such that $(W_i,\omega_i)$ and $(W_{i+1},\omega_{i+1})$ yield symplectomorphic fillings of $(M,\xi)$ after round symplectic 1-handle attachment along $L^i_-,L^i_+$ or $L^{i+1}_-,L^{i+1}_+$.
\end{corollary}

Corollary \ref{cor:overcounting} provides just one answer to the following question: under what conditions do fillings $(W_i,\omega_i)$ and $(W_{i+1},\omega_{i+1})$ of $(L(p_i,q_i),\xi_i)$ and $(L(p_{i+1},q_{i+1}),\xi_{i+1})$, respectively,  yield the same filling of $(M,\xi)$ after round symplectic 1-handle attachment?  A full answer to this question, along with a classification of the symplectic fillings of tight lens spaces, will complete the classification of fillings for virtually overtwisted contact structures on torus bundles.

\bibliographystyle{alpha}
\bibliography{../../references}
\end{document}

%% file: torus-bundle-filling.pdf_tex
\begingroup%
  \makeatletter%
  \providecommand\color[2][]{%
    \errmessage{(Inkscape) Color is used for the text in Inkscape, but the package 'color.sty' is not loaded}%
    \renewcommand\color[2][]{}%
  }%
  \providecommand\transparent[1]{%
    \errmessage{(Inkscape) Transparency is used (non-zero) for the text in Inkscape, but the package 'transparent.sty' is not loaded}%
    \renewcommand\transparent[1]{}%
  }%
  \providecommand\rotatebox[2]{#2}%
  \newcommand*\fsize{\dimexpr\f@size pt\relax}%
  \newcommand*\lineheight[1]{\fontsize{\fsize}{#1\fsize}\selectfont}%
  \ifx\svgwidth\undefined%
    \setlength{\unitlength}{589.11758062bp}%
    \ifx\svgscale\undefined%
      \relax%
    \else%
      \setlength{\unitlength}{\unitlength * \real{\svgscale}}%
    \fi%
  \else%
    \setlength{\unitlength}{\svgwidth}%
  \fi%
  \global\let\svgwidth\undefined%
  \global\let\svgscale\undefined%
  \makeatother%
  \begin{picture}(1,0.16833013)%
    \lineheight{1}%
    \setlength\tabcolsep{0pt}%
    \put(0,0){\includegraphics[width=\unitlength,page=1]{torus-bundle-filling.pdf}}%
    \put(0.12992069,0.0066786){\color[rgb]{0,0,0}\makebox(0,0)[lt]{\lineheight{1.25}\smash{\begin{tabular}[t]{l}$K_i$\end{tabular}}}}%
    \put(0.2853502,0.00797991){\color[rgb]{0,0,0}\makebox(0,0)[lt]{\lineheight{1.25}\smash{\begin{tabular}[t]{l}$K_{i-1}$\end{tabular}}}}%
    \put(0.65686468,0.0067162){\color[rgb]{0,0,0}\makebox(0,0)[lt]{\lineheight{1.25}\smash{\begin{tabular}[t]{l}$K_{i+1}$\end{tabular}}}}%
    \put(0.83877433,0.00673506){\color[rgb]{0,0,0}\makebox(0,0)[lt]{\lineheight{1.25}\smash{\begin{tabular}[t]{l}$K_i$\end{tabular}}}}%
  \end{picture}%
\endgroup%

%% file: lens-space-filling.pdf_tex
\begingroup%
  \makeatletter%
  \providecommand\color[2][]{%
    \errmessage{(Inkscape) Color is used for the text in Inkscape, but the package 'color.sty' is not loaded}%
    \renewcommand\color[2][]{}%
  }%
  \providecommand\transparent[1]{%
    \errmessage{(Inkscape) Transparency is used (non-zero) for the text in Inkscape, but the package 'transparent.sty' is not loaded}%
    \renewcommand\transparent[1]{}%
  }%
  \providecommand\rotatebox[2]{#2}%
  \newcommand*\fsize{\dimexpr\f@size pt\relax}%
  \newcommand*\lineheight[1]{\fontsize{\fsize}{#1\fsize}\selectfont}%
  \ifx\svgwidth\undefined%
    \setlength{\unitlength}{630.12110913bp}%
    \ifx\svgscale\undefined%
      \relax%
    \else%
      \setlength{\unitlength}{\unitlength * \real{\svgscale}}%
    \fi%
  \else%
    \setlength{\unitlength}{\svgwidth}%
  \fi%
  \global\let\svgwidth\undefined%
  \global\let\svgscale\undefined%
  \makeatother%
  \begin{picture}(1,0.13764468)%
    \lineheight{1}%
    \setlength\tabcolsep{0pt}%
    \put(0,0){\includegraphics[width=\unitlength,page=1]{lens-space-filling.pdf}}%
    \put(0.11710542,0.00421751){\color[rgb]{0,0,0}\makebox(0,0)[lt]{\lineheight{1.25}\smash{\begin{tabular}[t]{l}$K'_i$\end{tabular}}}}%
    \put(0.29574771,0.00543414){\color[rgb]{0,0,0}\makebox(0,0)[lt]{\lineheight{1.25}\smash{\begin{tabular}[t]{l}$K_{i-1}$\end{tabular}}}}%
    \put(0.64308683,0.00425266){\color[rgb]{0,0,0}\makebox(0,0)[lt]{\lineheight{1.25}\smash{\begin{tabular}[t]{l}$K_{i+1}$\end{tabular}}}}%
    \put(0.81792017,0.0042703){\color[rgb]{0,0,0}\makebox(0,0)[lt]{\lineheight{1.25}\smash{\begin{tabular}[t]{l}$K''_i$\end{tabular}}}}%
    \put(0,0){\includegraphics[width=\unitlength,page=2]{lens-space-filling.pdf}}%
  \end{picture}%
\endgroup%

%% file: farey.tex
\begin{tikzpicture}[scale=3]
\draw (0,0) circle (1);

\draw (-1,0) -- (1,0);
\draw (0:1) to[out=180,in=270]
	  (90:1) to[out=270,in=0]
	  (180:1) to[out=0,in=90]
	  (270:1) to[out=90,in=180]
	  (0:1);
\draw (0:1) to[out=180,in=225]
	  (45:1) to[out=225,in=270]
	  (90:1) to[out=270,in=315]
	  (135:1) to[out=315,in=0]
	  (180:1) to[out=0,in=45]
	  (225:1) to[out=45,in=90]
	  (270:1) to[out=90,in=135]
	  (315:1) to[out=135,in=180]
	  (0:1);
\draw (0:1) to[out=180,in=202.5]
	  (22.5:1) to[out=202.5,in=225]
	  (45:1) to[out=225,in=247.5]
	  (67.5:1) to[out=247.5,in=270]
	  (90:1) to[out=270,in=292.5]
	  (112.5:1) to[out=292.5,in=315]
	  (135:1) to[out=315,in=337.5]
	  (157.5:1) to[out=337.5,in=0]
	  (180:1) to[out=0,in=22.5]
	  (202.5:1) to[out=22.5,in=45]
	  (225:1) to[out=45,in=67.5]
	  (247.5:1) to[out=67.5,in=90]
	  (270:1) to[out=90,in=112.5]
	  (292.5:1) to[out=112.5,in=135]
	  (315:1) to[out=135,in=157.5]
	  (337.5:1) to[out=157.5,in=180]
	  (0:1);
	  
\draw[right] (1,0) node {$\frac{0}{1}$};
\draw[above] (0,1) node {$\frac{1}{1}$};
\draw[left] (-1,0) node {$\frac{1}{0}$};
\draw[below] (0,-1) node {$-\frac{1}{1}$};
\draw[above=2,right] (45:1) node {$\frac{1}{2}$};
\draw[above=2,left] (135:1) node {$\frac{2}{1}$};
\draw[below=2,right] (-45:1) node {$-\frac{1}{2}$};
\draw[below=2,left] (-135:1) node {$-\frac{2}{1}$};
\end{tikzpicture}

%% file: 1a.tex
\begin{tikzpicture}[scale=1.75]
\draw[thin] (0,-1) -- (0,1);
\draw[thin] (-1,0) -- (1,0);

\draw[thick] (-.894,-.447) -- (.894,.447);
\node[above] at (.894,.42) {$\mu_1$};

\draw[thick, dashed] (-1,0) -- (1,0);
\node[above] at (-1,0) {$\Gamma_1$};

\draw[thick, dashed] (0,-1) -- (0,1);
\node[left] at (0,-1) {$\Gamma_0$};

\draw[thick] (-.447,.894) -- (.447,-.894);
\node[right] at (.447,-.894) {$\mu_0$};

\draw [thick,domain=26.62:296.62] plot ({0.5*cos(\x)}, {0.5*sin(\x)});
\end{tikzpicture}

%% file: 1b.tex
\begin{tikzpicture}[scale=1.75]
\draw[thin] (0,-1) -- (0,1);
\draw[thin] (-1,0) -- (1,0);

\draw[thick] (-.894,-.447) -- (.894,.447);
\node[above] at (.894,.42) {$\mu_1$};

\draw[thick, dashed] (-1,0) -- (1,0);
\node[above] at (-1,0) {$\Gamma_1$};

\draw[thick, dashed] (0,-1) -- (0,1);
\node[left] at (0,-1) {$\Gamma_0$};

\draw[thick] (-.316,.949) -- (.316,-.949);
\node[right] at (.447,-.894) {$\mu_0$};

\draw [thick,domain=26.62:288.42] plot ({0.5*cos(\x)}, {0.5*sin(\x)});
\end{tikzpicture}

%% file: 2a-m-pos.tex
\begin{tikzpicture}[scale=2]
\draw[thin] (0,-1) -- (0,1);
\draw[thin] (-1,0) -- (1,0);

\draw[thick] (.707,.707) -- (-.707,-.707);
\node at (.8,.8) {$\mu_{1/2}$};

\draw[thick, dashed] (0,-1) -- (0,1);
\node[right] at (0,1) {$\Gamma_{1/2}$};

\draw[thick, dashed] (-.894,-.447) -- (.894,.447);
\node[below] at (-.9,-.4) {$\Gamma_{-1/2}$};

\draw[thick] (-.928,-.371) -- (.928,.371);
\node[below] at (1,.25) {$\mu_{-1/2}$};

\draw [thick,domain=45:360] plot ({0.5*cos(\x)}, {0.5*sin(\x)});
\draw [thick,domain=0:21.87] plot ({0.5*cos(\x)}, {0.5*sin(\x)});
\end{tikzpicture}

%% file: 2a-m-neg.tex
\begin{tikzpicture}[scale=2]
\draw[thin] (0,-1) -- (0,1);
\draw[thin] (-1,0) -- (1,0);

\draw[thick] (-.707,.707) -- (.707,-.707);
\node at (-.8,.8) {$\mu_{1/2}$};

\draw[thick, dashed] (0,-1) -- (0,1);
\node[right] at (0,-.8) {$\Gamma_{1/2}$};

\draw[thick, dashed] (-.894,-.447) -- (.894,.447);
\node[below] at (.95,.35) {$\Gamma_{-1/2}$};

\draw[thick] (-.707,-.707) -- (.707,.707);
\node[above] at (.8,.7) {$\mu_{-1/2}$};

\draw [thick,domain=135:360] plot ({0.5*cos(\x)}, {0.5*sin(\x)});
\draw [thick,domain=0:45] plot ({0.5*cos(\x)}, {0.5*sin(\x)});
\end{tikzpicture}

%% file: 2c-m-0.tex
\begin{tikzpicture}[scale=1.75]
\draw[thin] (0,-1) -- (0,1);
\draw[thin] (-1,0) -- (1,0);

\draw[thick] (0,-1) -- (0,1);
\node[right] at (0,1) {$\mu_1$};

\draw[thick, dashed] (-.707,.707) -- (.707,-.707);
\node[above] at (-.55,.61) {$\Gamma_1$};

\draw[thick, dashed] (.949,-.316) -- (-.949,.316);
\node[below] at (-.9,.3) {$\Gamma_0$};

\draw[thick] (.894,-.447) -- (-.894,.447);
\node[below] at (.9,-.42) {$\mu_0$};

\draw [thick,domain=90:333.38] plot ({0.5*cos(\x)}, {0.5*sin(\x)});
\end{tikzpicture}

%% file: 2c-k-minus.tex
\begin{tikzpicture}[scale=1.75]
\draw[thin] (0,-1) -- (0,1);
\draw[thin] (-1,0) -- (1,0);

\draw[thick] (.447,-.894) -- (-.447,.894);
\node[above] at (-.447,.894) {$\mu_1$};

\draw[thick, dashed] (-.707,.707) -- (.707,-.707);
\node[above] at (-.6,.6) {$\Gamma_1$};

\draw[thick, dashed] (.949,-.316) -- (-.949,.316);
\node[below] at (-.9,.3) {$\Gamma_0$};

\draw[thick] (-.928,.371) -- (.928,-.371);
\node[below] at (.928,-.371) {$\mu_0$};

\draw [thick,domain=116.6:338.13] plot ({0.5*cos(\x)}, {0.5*sin(\x)});
\end{tikzpicture}

%% file: 2c-geq-2.tex
\begin{tikzpicture}[scale=1.75]
\draw[thin] (0,-1) -- (0,1);
\draw[thin] (-1,0) -- (1,0);

\draw[thick] (.84,-.42) -- (-.894,.447);
\node[above] at (-.894,.42) {$\mu_1$};

\draw[thick, dashed] (-.707,.707) -- (.707,-.707);
\node[below] at (.45,-.45) {$\Gamma_1$};

\draw[thick, dashed] (.949,-.316) -- (-.949,.316);
\node[below] at (1.05,-.28) {$\Gamma_0$};

\draw[thick] (-.970,.243) -- (.970,-.243);
\node[above] at (.970,-.243) {$\mu_0$};

\draw [thick,domain=153.4:345.93] plot ({0.5*cos(\x)}, {0.5*sin(\x)});
\end{tikzpicture}

%% file: cfb.tex
\begin{tikzpicture}[scale=3]
\draw (0,0) circle (1);

\draw (-1,0) -- (1,0);
\draw (0:1) to[out=180,in=270]
	  (90:1) to[out=270,in=0]
	  (180:1) to[out=0,in=90]
	  (270:1) to[out=90,in=180]
	  (0:1);
\draw (0:1) to[out=180,in=225]
	  (45:1) to[out=225,in=270]
	  (90:1) to[out=270,in=315]
	  (135:1) to[out=315,in=0]
	  (180:1) to[out=0,in=45]
	  (225:1) to[out=45,in=90]
	  (270:1) to[out=90,in=135]
	  (315:1) to[out=135,in=180]
	  (0:1);
\draw (0:1) to[out=180,in=202.5]
	  (22.5:1) to[out=202.5,in=225]
	  (45:1) to[out=225,in=247.5]
	  (67.5:1) to[out=247.5,in=270]
	  (90:1) to[out=270,in=292.5]
	  (112.5:1) to[out=292.5,in=315]
	  (135:1) to[out=315,in=337.5]
	  (157.5:1) to[out=337.5,in=0]
	  (180:1) to[out=0,in=22.5]
	  (202.5:1) to[out=22.5,in=45]
	  (225:1) to[out=45,in=67.5]
	  (247.5:1) to[out=67.5,in=90]
	  (270:1) to[out=90,in=112.5]
	  (292.5:1) to[out=112.5,in=135]
	  (315:1) to[out=135,in=157.5]
	  (337.5:1) to[out=157.5,in=180]
	  (0:1);
	  
\draw[very thick] (-15:1) arc(-15:165:1);
\draw[very thick] (-15:0.96) -- (-15:1.04);
\draw[very thick] (165:0.96) -- (165:1.04);
	  
\draw[right] (1,0) node {$\frac{0}{1}$};
\draw[above] (0,1) node {$\frac{1}{1}$};
\draw[left] (-1,0) node {$\frac{1}{0}$};
\draw[below] (0,-1) node {$-\frac{1}{1}$};
\draw[above=2,left] (135:1) node {$\frac{2}{1}$};
\draw[above=2,left] (157.5:1) node {$\frac{3}{1}$};
\draw[below=2,right] (-45:1) node {$-\frac{1}{2}$};
\draw[below=2,right] (-22.5:1) node {$-\frac{1}{3}$};

\draw[below,right] (-15:1) node {$\Lambda^s$};
\draw[above,left] (165:1) node {$\Lambda^u$};
\end{tikzpicture}

%% file: interval.tex
\begin{tikzpicture}[xscale=10,yscale=1.8]
\draw [thick] (0,0) -- (1,0)
	  (0,-0.2) -- (0,0.2)
	  (0,0.2) -- (0.02,0.2)
	  (0,-0.2) -- (0.02,-0.2)
	  (0.1428,-0.1) -- (0.1428,0.1)
	  (0.2857,-0.2) -- (0.2857,0.2)
	  (0.2657,-0.2) -- (0.3057,-0.2)
	  (0.2657,0.2) -- (0.3057,0.2)
	  (0.4285,-0.1) -- (0.4285,0.1)
	  (0.5714,-0.1) -- (0.5714,0.1)
	  (0.7142,-0.2) -- (0.7142,0.2)
	  (0.6942,-0.2) -- (0.7342,-0.2)
	  (0.6942,0.2) -- (0.7342,0.2)
	  (0.8571,-0.1) -- (0.8571,0.1)
	  (0.98,-0.2) -- (1,-0.2)
	  (0.98,0.2) -- (1,0.2)
	  (1,-0.2) -- (1,0.2);
	  
\node [above] at (0,0.2) {$-1$};
\node [above] at (0.1428,0.2) {$-2$};
\node [above] at (0.2857,0.2) {$-3$};
\node [above] at (0.4285,0.2) {$\frac{-13}{4}$};
\node [above] at (0.5714,0.2) {$\frac{-23}{7}$};
\node [above] at (0.7142,0.2) {$\frac{-33}{10}$};
\node [above] at (0.8571,0.2) {$\frac{-76}{23}$};
\node [above] at (1,0.2) {$\frac{-119}{36}$};

\node [below] at (0,-0.2) {$0$};
\node [below] at (1,-0.2) {$1$};
\end{tikzpicture}

%% file: example-filling.pdf_tex
\begingroup%
  \makeatletter%
  \providecommand\color[2][]{%
    \errmessage{(Inkscape) Color is used for the text in Inkscape, but the package 'color.sty' is not loaded}%
    \renewcommand\color[2][]{}%
  }%
  \providecommand\transparent[1]{%
    \errmessage{(Inkscape) Transparency is used (non-zero) for the text in Inkscape, but the package 'transparent.sty' is not loaded}%
    \renewcommand\transparent[1]{}%
  }%
  \providecommand\rotatebox[2]{#2}%
  \newcommand*\fsize{\dimexpr\f@size pt\relax}%
  \newcommand*\lineheight[1]{\fontsize{\fsize}{#1\fsize}\selectfont}%
  \ifx\svgwidth\undefined%
    \setlength{\unitlength}{754.11765271bp}%
    \ifx\svgscale\undefined%
      \relax%
    \else%
      \setlength{\unitlength}{\unitlength * \real{\svgscale}}%
    \fi%
  \else%
    \setlength{\unitlength}{\svgwidth}%
  \fi%
  \global\let\svgwidth\undefined%
  \global\let\svgscale\undefined%
  \makeatother%
  \begin{picture}(1,0.1388957)%
    \lineheight{1}%
    \setlength\tabcolsep{0pt}%
    \put(0,0){\includegraphics[width=\unitlength,page=1]{example-filling.pdf}}%
    \put(0.23278759,0.12460955){\color[rgb]{0,0,0}\makebox(0,0)[lt]{\lineheight{1.25}\smash{\begin{tabular}[t]{l}$K_4$\end{tabular}}}}%
    \put(0,0){\includegraphics[width=\unitlength,page=2]{example-filling.pdf}}%
    \put(0.38876022,0.12460955){\color[rgb]{0,0,0}\makebox(0,0)[lt]{\lineheight{1.25}\smash{\begin{tabular}[t]{l}$K_3$\end{tabular}}}}%
    \put(0.56379922,0.12460955){\color[rgb]{0,0,0}\makebox(0,0)[lt]{\lineheight{1.25}\smash{\begin{tabular}[t]{l}$K_2$\end{tabular}}}}%
    \put(0.71894739,0.12460955){\color[rgb]{0,0,0}\makebox(0,0)[lt]{\lineheight{1.25}\smash{\begin{tabular}[t]{l}$K_1$\end{tabular}}}}%
    \put(0.86216114,0.12460955){\color[rgb]{0,0,0}\makebox(0,0)[lt]{\lineheight{1.25}\smash{\begin{tabular}[t]{l}$K_0$\end{tabular}}}}%
    \put(0.24100181,0.00327239){\color[rgb]{0,0,0}\makebox(0,0)[lt]{\lineheight{1.25}\smash{\begin{tabular}[t]{l}$-4$\end{tabular}}}}%
    \put(0.40211723,0.00327239){\color[rgb]{0,0,0}\makebox(0,0)[lt]{\lineheight{1.25}\smash{\begin{tabular}[t]{l}$-2$\end{tabular}}}}%
    \put(0.56721084,0.00327239){\color[rgb]{0,0,0}\makebox(0,0)[lt]{\lineheight{1.25}\smash{\begin{tabular}[t]{l}$-2$\end{tabular}}}}%
    \put(0.72434807,0.00327239){\color[rgb]{0,0,0}\makebox(0,0)[lt]{\lineheight{1.25}\smash{\begin{tabular}[t]{l}$-5$\end{tabular}}}}%
    \put(0.86756197,0.00327239){\color[rgb]{0,0,0}\makebox(0,0)[lt]{\lineheight{1.25}\smash{\begin{tabular}[t]{l}$-4$\end{tabular}}}}%
  \end{picture}%
\endgroup%

%% file: example-diagram.tex
\begin{tikzpicture}[xscale=14,yscale=1.8]
\draw [thick] (0,0) -- (1,0)
	  (0,-0.2) -- (0,0.2)
	  (0,0.2) -- (0.02,0.2)
	  (0,-0.2) -- (0.02,-0.2)
	  (0.1428,-0.1) -- (0.1428,0.1)
	  (0.2857,-0.2) -- (0.2857,0.2)
	  (0.2657,-0.2) -- (0.3057,-0.2)
	  (0.2657,0.2) -- (0.3057,0.2)
	  (0.4285,-0.1) -- (0.4285,0.1)
	  (0.5714,-0.1) -- (0.5714,0.1)
	  (0.7142,-0.2) -- (0.7142,0.2)
	  (0.6942,-0.2) -- (0.7342,-0.2)
	  (0.6942,0.2) -- (0.7342,0.2)
	  (0.8571,-0.1) -- (0.8571,0.1)
	  (0.98,-0.2) -- (1,-0.2)
	  (0.98,0.2) -- (1,0.2)
	  (1,-0.2) -- (1,0.2);
	  
\node [above] at (0,0.2) {$-1$};
\node [above] at (0.1428,0.2) {$-2$};
\node [above] at (0.2857,0.2) {$-3$};
\node [above] at (0.4285,0.2) {$\frac{-13}{4}$};
\node [above] at (0.5714,0.2) {$\frac{-23}{7}$};
\node [above] at (0.7142,0.2) {$\frac{-33}{10}$};
\node [above] at (0.8571,0.2) {$\frac{-76}{23}$};
\node [above] at (1,0.2) {$\frac{-119}{36}$};

\node [below] at (0,-0.2) {$0$};
\node [below] at (1,-0.2) {$1$};

\node [below] at (0,-0.6) {$L(43,13)$};
\node [below] at (0,-0.4) {$L(49,34)$};

\node [below] at (0.1428,-0.4) {$L(43,13)$};

\node [below] at (0.2857,-0.6) {$L(37,10)$};
\node [below] at (0.2857,-0.4) {$L(43,13)$};

\node [below] at (0.4285,-0.4) {$L(37,10)$};

\node [below] at (0.5714,-0.4) {$L(37,10)$};

\node [below] at (0.7142,-1) {$L(49,34)$};
\node [below] at (0.7142,-0.8) {$L(127,71)$};
\node [below] at (0.7142,-0.6) {$L(123,26)$};
\node [below] at (0.7142,-0.4) {$L(37,10)$};

\node [below] at (0.8571,-0.4) {$L(49,34)$};
\end{tikzpicture}